\documentclass[a4paper,10pt]{article}

\usepackage{amsmath}
\usepackage{amscd}
\usepackage{amsfonts}
\usepackage{stmaryrd}
\usepackage{mathrsfs}
\usepackage{amsthm}
\usepackage{ dsfont }
\usepackage{latexsym, amssymb,amsthm,amsmath,lscape,epsfig,setspace,tikz,slashed}
\usepackage[all]{xy}
\usepackage[retainorgcmds]{IEEEtrantools}
\usetikzlibrary{decorations.markings}
 \usetikzlibrary{arrows}
\usepackage[pdftex]{hyperref}
\usepackage[margin=10pt,font=small,labelfont=bf]{caption}
\usepackage{a4wide}

\newcommand{\mft}{\mathfrak{t}}

\newcommand{\mfg}{\mathfrak{g}}

\newcommand{\mfa}{\mathfrak{a}}

\newcommand{\R}{\mathbb{R}}
\newcommand{\Z}{\mathbb{Z}}

\newcommand{\J}{\mathcal{J}}

\newcommand{\T}{\mathbb{T}}

\newcommand{\C}{\mathbb{C}}
\newcommand{\mbP}{\mathbb{P}}
\newcommand{\mcL}{\mathcal{L}}

\newcommand{\mcG}{\mathcal{G}}
\newcommand{\lb}{\llbracket}
\newcommand{\rb}{\rrbracket}

\newcommand{\mfB}{\mathfrak{B}}

\numberwithin{equation}{section}

\theoremstyle{plain}
\newtheorem{thm}{Theorem}[section]
\newtheorem{prop}[thm]{Proposition}

\theoremstyle{definition}
\newtheorem{lem}[thm]{Lemma}
\newtheorem{defn}[thm]{Definition}
\newtheorem{ex}[thm]{Example}
\newtheorem{cor}[thm]{Corollary}

\theoremstyle{remark}
\newtheorem{rem}[thm]{Remark}

\begin{document}

\title{Blow-ups in generalized K\"ahler geometry}

\author{
J.L. van der Leer Dur\'an \thanks{{\tt j.l.vanderleerduran@uu.nl}} \\
       Department of Mathematics\\
Utrecht University\\
}

\date{\vspace{-5ex}}

\maketitle

\abstract{}
\noindent We continue the study of blow-ups in generalized complex geometry with the blow-up theory for generalized K\"ahler manifolds. The natural candidates for submanifolds to be blown-up are those which are generalized Poisson for one of the two generalized complex structures and can be blown up in a generalized complex manner. We show that the bi-Hermitian structure underlying the generalized K\"ahler pair lifts to a degenerate bi-Hermitian structure on this blow-up. Then, using a deformation procedure based on potentials in K\"ahler geometry, we identify two concrete situations in which one can deform the degenerate structure on the blow-up into a non-degenerate one. We end with an investigation of generalized K\"ahler Lie groups and give a concrete example on $(S^1)^n\times (S^3)^m$, for $n+m$ even.
\vskip12pt

\tableofcontents

\section{Introduction}

Generalized K\"ahler geometry was born in $1984$ when Gates, Hull and Ro\^cek \cite{MR776369} discovered 
new supersymmetric sigma-models in physics, whose background geometry could be relaxed from K\"ahler to generalized K\"ahler. At that time it appeared in the guise of bi-Hermitian geometry; pairs of complex structures $(I_+,I_-)$ compatible with a common metric $g$, satisfying an additional integrability equation. Bi-Hermitian geometry then gained interest from mathematicians and some theory was developed, notably in dimension four. For instance, a classification of four--manifolds carrying two compatible complex structures $(I_+,I_-)$ with $I_+\neq \pm I_-$ has been obtained in the cases where the metric is anti-self-dual (Pontecorvo \cite{MR1467652}) and where the first Betti number is even (Apostolov, Gauduchon and Grantcharov \cite{MR1702248}).
\
\newline
\newline
Around $2003$ generalized K\"ahler geometry appeared in a different formulation out of the work of Gualtieri \cite{MR2811595}, in the context of generalized complex geometry. Generalized geometry is centered on the idea of replacing the tangent bundle of a manifold by the sum of its tangent and cotangent bundle. This creates enough room for merging both complex and symplectic structures into one object called a generalized complex structure. Just as a K\"ahler manifold consists out of a complex and a symplectic structure which are compatible, a generalized K\"ahler manifold is defined by a compatible pair of generalized complex structures. Gualtieri showed that generalized K\"ahler geometry is equivalent to bi-Hermitian geometry, providing an interesting new point of view on the latter and resulting in new advances in the theory. For instance, a reduction theory became available for generalized K\"ahler structures, with applications to moduli spaces of instantons (Hitchin \cite{MR2217300}; Burzstyn, Cavalcanti, Gualtieri \cite{MR2397619}, \cite{BCGinstantons}). Another example is the deformation theorem of Goto \cite{MR2669364}, which states that on a compact manifold, a deformation of one of the two structures in a generalized K\"ahler pair can be coupled to a deformation of the second, provided the second is of a special type (``generalized Calabi-Yau"). This theorem can be applied to compact K\"ahler manifolds with a holomorphic Poisson structure, giving an important class of examples. 

Despite all these developments the study of generalized K\"ahler manifolds remains difficult, and examples where the underlying manifold does not support a K\"ahler structure are scarce. Noteworthy examples of the latter include even dimensional compact Lie groups (Gualtieri \cite{gualtieri-2010}) and some specific solvmanifolds (Fino, Tomassini \cite{MR2496412}).
\newline
\newline
\noindent In this paper we will be concerned with the blow-up theory of generalized K\"ahler manifolds. Blow-ups in generalized complex geometry were studied in \cite{MR2574746} and \cite{GCblowups}, leading to the following conclusions. There are two types of generalized complex submanifolds that are suitable for blowing up; generalized Poisson submanifolds and generalized Poisson transversals. The former look complex in transverse directions and, under an additional hypothesis on the Lie algebra induced on the conormal bundle, such submanifolds admit a (canonical) blow-up. Generalized Poisson transversals are symplectic in transverse directions and admit a global normal form, which implies the existence of a (non-canonical) blow-up. In a generalized K\"ahler manifold a generalized Poisson submanifold for one of the structures is automatically a generalized Poisson transversal for the other, just as a complex manifold of a K\"ahler manifold is automatically symplectic. The main question of this paper is then whether such submanifolds admit a generalized K\"ahler blow-up. The case of a point in a four--dimensional manifold has been investigated in \cite{MR2861778}, and we will adapt the techniques employed there to higher dimensional submanifolds. We obtain two situations where a blow-up exists. The first is when the Lie algebra structure on the conormal bundle is Abelian, which geometrically means that the exceptional divisor of the blow-up is again generalized Poisson. The second situation is when the submanifold in question is contained in a Poisson divisor for one of the two complex structures in the bi-Hermitian picture. This will be the case for instance when the structure for which $Y$ is generalized Poisson is generically symplectic. We end the paper with an explicit investigation of generalized K\"ahler structures on even dimensional compact Lie groups. We show that a maximal torus, which can be taken generalized Poisson for a suitably chosen generalized K\"ahler structure, can be blown up in a generalized complex way if and only if the Lie group equals $(S^1)^n\times (S^3)^m$, with $n+m$ even. There is then no further constraint to blow-up the generalized K\"ahler structure.  
\
\newline
\newline
\noindent \textsl{Organization:} The paper is organized as follows. In Section \ref{09:13:21} we lay out the basic ingredients from generalized complex and generalized K\"ahler geometry. For a more detailed explanation including proofs we refer to \cite{MR2811595} and \cite{gualtieri-2010}. In Section \ref{17:46} we discuss a deformation procedure for bi-Hermitian structures for which the metric is possibly degenerate. This part is based on \cite{MR2371181} and \cite{MR2861778}. Then, in Section \ref{09:19:55} we show that the bi-Hermitian structure lifts to a degenerate structure on the blow-up and, under additional hypotheses, apply the deformation procedure from Section \ref{17:46} to obtain a generalized K\"ahler blow-up. Finally, in Section \ref{09:21:55} we investigate submanifolds of generalized K\"ahler Lie groups that are suitable for blowing up.
\newline
\newline
\textsl{Acknowledgements:} The author is thankful to Gil Cavalcanti for assistence and useful conversations. 
This research was supported by a Vidi grant for NWO, the Dutch Science Foundation.

\section{Generalized geometry}
\label{09:13:21}

\subsection{Generalized complex geometry}

Let $M$ be a real $2n$-dimensional manifold and $H$ a closed real $3$--form. In generalized geometry the tangent bundle is replaced by $\T M:=TM\oplus T^\ast M$. It is endowed with a \textsl{natural pairing}
\begin{align*}
\langle X+\xi,Y+\eta	\rangle:=\frac{1}{2}(\xi(Y)+\eta(X))
\end{align*}
and a bracket on its space of sections called the \textsl{Courant bracket}:
\begin{align*}
\lb X+\xi,Y+\eta	\rb:=[X,Y]+\mcL_X(\eta)-\iota_Yd\xi-\iota_Y\iota_XH.
\end{align*}
It satisfies the Jacobi identity but is not skew-symmetric. 
\begin{defn}
A \textsl{generalized complex structure} on $(M,H)$ is a complex structure $\J$ on $\T M$ which is orthogonal with respect to $\langle, \rangle$ and whose $(+i)$--eigenbundle $L\subset\T M_\C$ is involutive\footnote{A subbundle of $\T M$ is called involutive if its space of sections is closed with respect to the Courant bracket.}.
\end{defn}
\noindent
A Lagrangian, involutive subbundle $L\subset \T M_\C$ is also called a \textsl{Dirac structure}, and generalized complex structures correspond in a one-to-one fashion with Dirac structures $L$ satisfying the additional non-degeneracy condition $L\cap \overline{L}=0$. 
\begin{ex}
\label{11:37:30}
The main examples come from complex and symplectic geometry: if $I$ is a complex structure and $\omega$ a symplectic structure, then 
\begin{align}
\J_I:=
\begin{pmatrix}
-I & 0 \\
0 & I^\ast
\end{pmatrix},
\ \ \ \ \ 
\J_\omega:=
\begin{pmatrix}
0 & -\omega^{-1} \\
\omega & 0
\end{pmatrix},
\label{13:31}
\end{align}
are generalized complex structures, with associated Dirac structures $L_I=T^{0,1}\oplus (T^\ast)^{1,0}$ and $L_\omega=(1-i\omega)T$. Another important example is provided by a holomorphic Poisson structure $(I,\sigma)$. If $Q$ is the real part of $\sigma$, then 
\begin{align}
\J_{(I,\sigma)}=
\begin{pmatrix}
-I & 4IQ \\
0 & I^\ast
\end{pmatrix}
\label{13:30}
\end{align}
is generalized complex and $L_{(I,\sigma)}=T^{0,1}\oplus (1+\sigma)(T^\ast)^{1,0}$. In these examples $H=0$.
\end{ex}
\noindent Dirac structures can conveniently be described in terms of differential forms. There is a natural action of the Clifford algebra of $(\T M,\langle, \rangle)$ on forms given by 
\begin{align*}
(X+\xi)\cdot \rho=\iota_{{}_X}\rho+\xi\wedge \rho,
\end{align*} 
giving an identification between the space of differential forms and the space of spinors for $Cl(\T M,\langle,\rangle)$. 
A line subbundle $K\subset \Lambda^\bullet T^\ast M_\C$ gives rise to an isotropic subbundle $L\subset \T M_\C$ by taking the annihilator
\[
L=\{X+\xi\in \T M_\C|(X+\xi)\cdot K=0\}.
\]
This yields to a one-to-one correspondence between Dirac structures $L\subset \T M_\C$, and complex line bundles $K\subset \Lambda^\bullet T^\ast M_\C$ which satisfy the following two conditions. Firstly, $K$ has to be generated by \textsl{pure spinors}, i.e.\ forms $\rho$ which at each point $x$ admit a decomposition
\begin{align}
\rho_x=e^{B+i\omega}\wedge \Omega
\label{13:09:56}
\end{align}
where $B+i\omega$ is a $2$--form and $\Omega$ is decomposable. This condition is equivalent to $L$ being of maximal rank. Secondly, if $\rho$ is a local section of $K$ there should exist $X+\xi\in \Gamma(\T M_\C)$ with
\[d^H\rho=(X+\xi)\cdot \rho.	\]
This condition amounts to the involutivity of $L$. The condition $L\cap \bar{L}=0$ can then be expressed in spinor language using the \textsl{Chevalley pairing}: If $\rho\in \Gamma(K)$ is non-vanishing then
\begin{align*}
L\cap \bar{L}=0\Longleftrightarrow (\rho,\bar{\rho})_{{}_{Ch}}:=(\rho\wedge \bar{\rho}^T)_{top}\neq 0.
\end{align*}
Here the superscript $T$ stands for transposition, acting on a degree $l$--form by $(\beta_1 \ldots \beta_l)^T=\beta_l\ldots\beta_1$, and the subscript $top$ stands for the highest degree component. If $\rho$ is given by (\ref{13:09:56}) at a particular point $x$ then this condition becomes 
\begin{align}
\omega^{n-k}\wedge \Omega\wedge \bar{\Omega}\neq 0,
\label{14:35:04}
\end{align}
 where $k:=\text{deg}(\Omega)$. The line bundle $K$ associated to a generalized complex structure is called the \textsl{canonical line bundle}, and the integer $k$ appearing in (\ref{14:35:04}) is called the \textsl{type} at $x$. Structures of type $0$ and $n$ are called symplectic, respectively complex\footnote{Such structures are in fact equivalent to symplectic, respectively complex structures, where equivalence is defined in Definition \ref{12:24}}. Another description of the type is as follows. Every generalized complex structure naturally induces a Poisson structure given by the composition 
\begin{align}
\pi_\J:T^\ast M\hookrightarrow \T M \xrightarrow{\J} \T M \twoheadrightarrow TM.
\label{12:20}
\end{align}
The conormal bundle to the leaves, i.e.\ the kernel of $\pi_\J$, is given by the complex distribution 
\begin{align*}
\nu_\J:=T^\ast M \cap \J T^\ast M.
\end{align*}
In general $\nu_\J$ might be singular, for its complex dimension may jump in even steps from one point to the next. Intuitively, one can think of generalized complex structures as Poisson structures with transverse complex structures, and the type at a point $x$ equals the number of complex directions, i.e.\ 
\begin{align*}
\text{type}_x(\J)=\text{dim}_\C(\nu_\J)_x=\frac{1}{2}\text{corank}_\R(\pi_\J)_x.
\end{align*}
Since generalized geometry mixes covariant and contravariant tensors the notion of morphism requires some care. 
\begin{defn}
A \textsl{generalized map} between $(M_1,H_1)$ and $(M_2,H_2)$ is a pair $\Phi:=(\varphi,B)$, where $\varphi:M_1\rightarrow M_2$ is a smooth map and $B\in \Omega^2(M_1)$ satisfies $\varphi^\ast H_2=H_1+dB$. 
\end{defn}
\noindent
We will often abbreviate $(\varphi,0)$ by $\varphi$ and drop the prefix ``generalized". An important role is played by \textsl{$B$-field transformations}, maps of the form\footnote{The minus sign is chosen so that $e^B\wedge ((X+\xi)\cdot \rho)=(e^B_\ast(X+\xi))\cdot e^B\wedge\rho$.} $(Id,-B)=:e^B_\ast$. They act on $\T M$ via
\begin{align}
e^B_\ast:X+\xi\mapsto X+\xi-\iota_XB.
\label{09:25:06}
\end{align}
Given $u\in \Gamma(\T M)$ we denote by $\text{ad}(u):\Gamma(\T M)\rightarrow \Gamma(\T M)$ the adjoint action with respect to the Courant bracket. This infinitesimal symmetry has a flow, i.e.\ a family of isomorphisms $\psi_t:\T M\rightarrow \T M$ with
$
d/dt (\psi_t(v))=-\lb u ,\psi_t(v) \rb.
$
If $u=X+\xi$ and $\varphi_t$ is the flow of $X$ then
\begin{align}
\psi_t=
(\varphi_t)_\ast \circ e_\ast^{-B_t}.
\label{10:12:54}
\end{align}
where $B_t:=\int_0^t\varphi_r^\ast(d\xi+\iota_{{}_X}H)dr$. A map $\Phi=(\varphi,B)$ gives rise to a correspondence: we say that $X+\xi$ is $\Phi$-related to $Y+\eta$, and write 
$X +\xi \sim_{{}_\Phi} Y+\eta$, if 
\begin{align*}
\varphi_\ast X=Y, \ \ \xi=\varphi^\ast \eta -\iota_{{}_X}B.
\end{align*}
\begin{defn}
A map $\Phi:(M_1,H_1,\J_1)\rightarrow(M_2,H_2,\J_2)$ is called \textsl{generalized holomorphic} if for every $X+\xi\in \T M_1$ and $Y+\eta\in \T M_2$ with $X +\xi \sim_{{}_\Phi} Y+\eta$, we have $\J_1(X +\xi) \sim_{{}_\Phi}\J_2(Y+\eta)$.
If $\Phi$ is in addition invertible we call it an \textsl{isomorphism}. 
\label{12:24}
\end{defn}
\begin{rem}
It follows immediately from the definition that $\varphi$ is a Poisson map, i.e.\ $\varphi_\ast \pi_{\J_1}=\pi_{\J_2}$. This is quite restrictive, for example if the target is symplectic then $\varphi$ has to be a submersion. In the complex category we recover the usual notion of holomorphic maps.
\label{13:34}
\end{rem} 
\noindent
In case $\varphi$ is a diffeomorphism a more concrete description in terms of spinors can be given. If $K_i$ is the canonical bundle for $\J_i$, $\Phi$ being an isomorphism amounts to
\begin{align*}
K_1=e^B\wedge \varphi^\ast K_2.
\end{align*}
\noindent
We can now state the analogue of the Newlander-Nirenberg and Darboux Theorems in generalized complex geometry.
\begin{thm}[\cite{MR3128977}] Let $(M,H,\J)$ be a generalized complex manifold. If $x\in M$ is a point where $\J$ has type $k$, then a neighborhood of $x$ is isomorphic to a neighborhood of $(0,0)$ in 
\begin{align}
(\R^{2n-2k},\omega_{st}) \times (\C^k,\sigma):=(\R^{2n-2k}\times \C^k, \J_{\omega_{st}}\times \J_{(i,\sigma)})
\label{13:19}
\end{align} 
where $\omega_{st}$ is the standard symplectic form, $\sigma$ is a holomorphic Poisson structure which vanishes at $0$, and $\J_{\omega_{st}}$ and $\J_{(i,\sigma)}$ are defined in Example \ref{11:37:30}. 
\label{16:12:53}
\end{thm}
\noindent To blow up submanifolds we need an appropriate notion of generalized complex submanifold. The definition that we will use generalizes complex and symplectic submanifolds, as those are the submanifolds that are known to admit a blow up in their respective categories (there is for instance no blow-up available for (co-)isotropic submanifolds of a symplectic manifold). Let $\Phi=(\varphi,B)$ be a map and $L_2$ a Dirac structure on $(M_2,H_2)$. We define the \textsl{backward image} of $L_2$ along $\Phi$ by
\begin{align}
\label{12:18:32}
\mfB \Phi(L_2):=\{X+\varphi^\ast\xi-\iota_{{}_X}B | \ \varphi_\ast X+\xi\in L_2\}.
\end{align}
This is a Dirac structure on $(M_1,H_1)$, provided it is a smooth vector bundle. A sufficient condition for this is that $ker(d\varphi^\ast)\cap \varphi^\ast L$ is of constant rank. More information can be found in \cite{MR3098084}. 
\begin{defn}
A \textsl{generalized complex submanifold} is a submanifold $i:Y\hookrightarrow (M,H,\J)$ such that $\mfB i(L)$ is generalized complex, i.e.\ is smooth and $\mfB i(L)\cap \overline{\mfB i(L)}=0$. 
\label{20:22}
\end{defn}
\begin{rem}
A sufficient condition for smoothness is that $N^\ast Y\cap \J N^\ast Y$ is of constant rank. Moreover, the second condition is equivalent to $\J N^\ast Y \cap (N^\ast Y)^\perp \subset N^\ast Y$. In complex or symplectic manifolds we recover the usual notion of complex respectively symplectic submanifolds. Also, a point is always a generalized complex submanifold. Note that if $M$ is symplectic, $i$ is generalized holomorphic if and only if $Y$ is an open subset.  
\end{rem}






\subsection{Generalized K\"ahler geometry}
Generalized complex structures were introduced with the purpose of unifying complex and symplectic structures into one framework. On a K\"ahler manifold we have both a complex and a symplectic structure which are compatible with each other. Here is the generalized version.
\begin{defn}
A \textsl{generalized K\"ahler structure} on $M$ is a pair of commuting generalized complex structures $(\J_1,\J_2)$, such that $\mcG:=-\J_1\J_2$ defines a \textsl{generalized metric}, i.e.\  
\begin{align*}
(u,v)\mapsto \langle \mcG u,v	 \rangle=\langle \J_1u, \J_2v	  \rangle
\end{align*}
is a positive definite metric on $\T M$.  
\end{defn}
\begin{ex} \label{10:47:25}
A natural example is given by an ordinary K\"ahler manifold $(M,I,\omega)$. Define $\J_1:=\J_I$ and $\J_2:=-\J_\omega$, which commute because $I$ is compatible with $\omega$, and 
\begin{align*}
\mcG=\begin{pmatrix} 0 & g^{-1} \\ g & 0 \end{pmatrix}
\end{align*}
is indeed positive, where $g:=-\omega I$ is the associated K\"ahler metric. 
\end{ex}
\noindent Since $\mcG^2=1$, $\T M$ decomposes into its $(\pm 1)$--eigenspaces $V_+$ and $V_-$, on which the pairing restricts to a positive, respectively negative definite form. So choosing a generalized metric  amounts to choosing a reduction of structure groups for $\T M$ from $O(2n,2n)$ to $O(2n)\times O(2n)$. The restriction of $\J_1$ induces a complex structure on $V_\pm$, leading to decompositions
\begin{align*}
(V_\pm)_\C=V^{1,0}_\pm\oplus V^{0,1}_\pm.
\end{align*}
As $\J_2$ equals $\pm \J_1$ on $V_\pm$, we obtain
\begin{align}
L_1=V^{1,0}_+\oplus V^{1,0}_-, \hspace{15mm} L_2=V^{1,0}_+\oplus V^{0,1}_-. 
\label{11:55}
\end{align}
\noindent Because $V_\pm$ does not intersect the isotropics $TM$ and $T^\ast M$, the projection $V_\pm\rightarrow TM$ is an isomorphism and we can write $V_\pm$ as the graph of a map $a_\pm:TM\rightarrow T^\ast M$. If we decompose $a_+=g+b$ where $g$ is symmetric and $b$ is skew, then positivity of $\langle , \rangle|_{V_+} $ implies that $g$ is positive definite, while orthogonality of $V_+$ and $V_-$ implies $a_-=-g+b$. Transporting the complex structure given by $\J_1$ on $V_\pm$ to $TM$ we obtain two almost complex structures $I_\pm$ on $M$, both compatible with $g$. Such a tuple $(g,b,I_+,I_-)$ will be called an almost\footnote{The adjective ``almost" refers to a structure without assuming any integrability conditions. The appropriate integrability conditions in this case are given by Proposition \ref{11:28:39} $ii)$.} \textsl{bi-Hermitian structure}. The above construction can then be reversed, giving a bijection between almost bi-Hermitian structures $(g,b,I_+,I_-)$ and almost generalized K\"ahler structures $(\J_1,\J_2)$. Integrability can be expressed as follows. 
\begin{prop}[{\cite[Theorem 6.28]{gualtieri-2003}}] 
Let $(\J_1,\J_2)$ be an almost generalized K\"ahler structure and $(g,b,I_+,I_-)$ the associated almost bi-Hermitian structure. Then the following are equivalent:
\begin{itemize}
\item[i)] $(\J_1,\J_2)$ is generalized K\"ahler.
\item[ii)] $I_\pm$ are both integrable and $\pm d^c_\pm\omega_\pm=H-db$, where $\omega_\pm=gI_\pm$ and $d^c_\pm=i(\bar{\partial}_\pm-\partial_\pm)$.
\item[iii)] $I_\pm$ are both integrable and $\nabla^\pm I_\pm =0$, where $\nabla^\pm:=\nabla\mp\frac{1}{2}g^{-1}(H-db)$ and $\nabla$ is the Levi-Cevita connection associated to $g$. 
\end{itemize}
\label{11:28:39} 
\end{prop}
\noindent An explicit relation between $(\J_1,\J_2)$ and $(g,b,I_\pm)$ is given by
\begin{align}
\J_1&=\frac{1}{2}
\begin{pmatrix}
1 & 0 \\ 
b & 1 \\
\end{pmatrix}
\begin{pmatrix}
I_++I_- & 	-(\omega^{-1}_+-\omega^{-1}_-)	\\
\omega_+-\omega_-& -(I^\ast_++I^\ast_-)
\end{pmatrix}
\begin{pmatrix}
1 & 0 \\ 
-b & 1 \\
\end{pmatrix},
\nonumber
\\
\J_2&=\frac{1}{2}
\begin{pmatrix}
1 & 0 \\ 
b & 1 \\
\end{pmatrix}
\begin{pmatrix}
I_+-I_- & 	-(\omega^{-1}_++\omega^{-1}_-)	\\
\omega_++\omega_-& -(I^\ast_+-I^\ast_-)
\end{pmatrix}
\begin{pmatrix}
1 & 0 \\ 
-b & 1 \\
\end{pmatrix}.
\label{10:42:52}
\end{align}
From this we see that 
\begin{align}
\label{09:05} \pi_{\J_1}=-\frac{1}{2}(\omega_+^{-1}-\omega_-^{-1}), \ \ \ \ \ 
\pi_{\J_2}=-\frac{1}{2}(\omega_+^{-1}+\omega_-^{-1}).
\end{align}
\noindent It follows that $\pi_{\J_1} + \pi_{\J_2}$ is invertible and, using a little bit of linear algebra, this implies that 
\begin{align*}
\text{type}(\J_1)+\text{type}(\J_2)\leq n.
\end{align*}
One can also relate the parity of the types of $\J_1$ and $\J_2$ to the orientations of $I_\pm$. In general, on a $2n$--dimensional manifold, $\text{type}(\J_1)=n \ \text{mod} \ 2$ if and only if $I_+$ and $I_-$ induce the same orientation, while $\text{type}(\J_2)= n \ \text{mod} \ 2$ if and only if $I_+$ and $-I_-$ induce the same orientations. 
\begin{rem} 
\
\begin{itemize}
\item[1] If $(\J_1,\J_2)$ is generalized K\"ahler then so is $(\J_2,\J_1)$, with the same generalized metric $\mcG$. So when considering e.g. a generalized Poisson submanifold (c.f. Section \ref{11:06:22}) for one of the two structures, we may as well assume this to be $\J_1$.
\item[2] We can always gauge away the two-form $b$ by a transformation of the form (\ref{09:25:06}), at the expense of modifying $H$ by $H-db$. In what follows we will often implicitly assume this has been done, and we will refer to the tuple $(g,I_\pm,H)$ as the bi-Hermitian structure.
\end{itemize}
\end{rem}
\noindent The difficult feature of bi-Hermitian geometry lies in the fact that $I_+$ and $I_-$ do not commute in general. Therefore, standard techniques in complex geometry, such as decomposition of forms into types, become difficult as they can be performed only for one of the two complex structures at a time. This failure of commutativity suggests that important information about the generalized K\"ahler structure is contained in the tensor 
\begin{align}
Q:=-\frac{1}{2}[I_+,I_-]g^{-1}:T^\ast M \rightarrow TM.
\label{09:06}
\end{align}
Since $Q$ is skew-symmetric we can regard it as a bivector, and it was observed in \cite{MR1702248} in the $4$--dimensional case and in \cite{MR2217300} in the general case, that $Q$ is Poisson. In fact, it turns out to be the real part of two holomorphic Poisson structures  
\begin{align}
\label{15:18:13}
\sigma_\pm:=Q-iI_\pm Q.
\end{align}
One can prove this directly in local coordinates using the integrability conditions (see \cite{MR2217300}), or in the following more abstract way (see \cite{gualtieri-2010}). If $L_1$ is a Dirac structure on $(M,H_1)$ and $L_2$ a Dirac structure on $(M,H_2)$, we can form their \textsl{Baer-sum} on $(M,H_1+H_2)$ via
\begin{align*}
L_1\boxtimes L_2:=\mfB i(L_1\times L_2)=\{X+\xi+\eta | X+\xi\in L_1, X+\eta\in L_2 	\}, 
\end{align*}
where $i=(i,0):(M,H_1+H_2) \rightarrow (M,H_1) \times (M,H_2)$ denotes the diagonal map, and the backward image is defined in (\ref{12:18:32}). A sufficient condition for $L_1\boxtimes L_2$ to be smooth is that $L_1\cap L_2\cap T^\ast M$ is of constant rank. Note that there is a natural map $L_1\times_{TM} L_2 \rightarrow L_1\boxtimes L_2$ given by $(X+\xi,X+\eta)\mapsto X+\xi+\eta$, which is an isomorphism if and only if $L_1\cap L_2\cap T^\ast M=0$. The latter condition also ensures that for spinors $\rho_1$, $\rho_2$ for $L_1$, $L_2$, the product $\rho_1\wedge \rho_2$ does not vanish and forms a spinor for $L_1\boxtimes L_2$. 

Observe that $TM$, considered as Dirac structure on $(M,0)$, acts as a two-sided identity with respect to the Baer-sum operation, i.e.\ $L\boxtimes TM=L=TM\boxtimes L$. There is also an inverse for each Dirac structure $L$ on $(M,H)$, given by $L^T:=\{ X+\xi | X-\xi\in L\}$ on $(M,-H)$, which satisfies $L\boxtimes L^T=L^T\boxtimes L=TM$. If $L$ comes from a generalized complex structure, a quick computation shows that  
\begin{align*}
L\boxtimes \bar{L}^T=\{-\frac{i}{2}\pi_{\J}(\xi)+\xi| \ \xi\in T^\ast M_\C\}=\text{graph}(-\frac{i}{2}\pi_{\J}).
\end{align*}
This is one way to see that $\pi_{\J}$ is integrable. In a similar spirit we have the following proposition which follows from the results in \cite{gualtieri-2010}, and we give a proof for completeness.  
\begin{prop}
Let $(\J_1,\J_2)$ be a generalized K\"ahler structure with associated Dirac structures $L_1$, $L_2$, and let $(g,I_\pm,H)$ be the corresponding bi-Hermitian structure. Then 
\begin{align*}
\bar{L}^T_1\boxtimes \overline{L_2}=L_{(I_+,-\frac{1}{8}\sigma_+)}, \hspace{15mm} \bar{L}^T_1\boxtimes L_2=L_{(I_-,-\frac{1}{8}\sigma_-)}, 
\end{align*}
where $L_{(I_\pm,-\frac{1}{8}\sigma_\pm)}$ are defined in Example \ref{11:37:30}. In particular, $\sigma_\pm$ are both holomorphic Poisson. 
\end{prop}
\begin{proof}
We will only show that $L_{(I_+,-\frac{1}{8}\sigma_+)}\subset \bar{L}^T_1\boxtimes \overline{L_2}$; equality then follows from dimensional reasons and the case of $\sigma_-$ is similar. We have
\begin{align*}
L_{(I_+,-\frac{1}{8}\sigma_+)}=\{X+\sigma_+(\xi)-8\xi | \ X\in T^{0,1}_+M, \xi \in T_+^{\ast 1,0}M	\},
\end{align*}
where $T^{1,0}_+M$ denotes $(+i)$--eigenspace for $I_+$. We can write $X=X-g(X)+g(X)$ 
and since $X+g(X)\in V_+^{0,1}=\overline{L}_1\cap \overline{L}_2$, we see that $X\in \bar{L}^T_1\boxtimes \overline{L_2}$. Next, let us denote by $P_\pm:=\frac{1}{2}(1-iI_\pm)$ the projections onto $T_\pm^{1,0}M$. A quick calculation yields 
\begin{align*}
\sigma_+=4g^{-1}\overline{P}_+\overline{P}_-P_+.
\end{align*} 
Using this we obtain, for $\xi\in T_+^{\ast 1,0}M$,  
\begin{align}
\label{14:18:26} \sigma_+(\xi)&=4g^{-1}(\xi-P_+\bar{P}_-\xi) -4g^{-1}(P_-\xi)\\
\label{14:18:37}&=-4g^{-1}(P_+\bar{P}_-\xi)+4g^{-1}(\bar{P}_-\xi).
\end{align}
Then (\ref{14:18:26}) and (\ref{14:18:37}) are decompositions of $\sigma_+(\xi)$ in $T^{0,1}_+M+T^{0,1}_-M$ and $T^{0,1}_+M+T^{1,0}_-M$ respectively. Writing $\zeta:=4(\xi-P_+\bar{P}_-\xi+P_-\xi)$ and $\eta:=-4P_+\bar{P}_-\xi-4\bar{P}_-\xi$ we have
\begin{align*}
\sigma_+(\xi)-8\xi =\sigma_+(\xi)-\zeta+\eta.
\end{align*} 
Equation (\ref{14:18:26}) implies that $\sigma_+(\xi)+\zeta \in \overline{L}_1$ while (\ref{14:18:37}) implies that $\sigma_+(\xi)+\eta\in \overline{L}_2$. In particular $\sigma_+(\xi)-8\xi\in \bar{L}_1^T\boxtimes \overline{L}_2$, so indeed  $L_{(I_+,-\frac{1}{8}\sigma_+)}\subset \bar{L}^T_1\boxtimes \overline{L_2}$.
\end{proof}
\noindent The fact that $\bar{L}^T_1\boxtimes \overline{L_2}$ is smooth can also be seen directly from $\bar{L}_1^T\cap \overline{L}_2\cap T^\ast M=V^{0,1}_+\cap T^\ast M=0$, which in addition shows that 
\begin{align}
\bar{\rho}_1^T\wedge \bar{\rho}_2=e^{-\frac{1}{8}\sigma_+}\Omega_+
\label{09:29}
\end{align}
where $\Omega$ is a suitably scaled $(n,0)$--form for $I_+$. Similarly,  
\begin{align}
\bar{\rho}_1^T\wedge \rho_2=e^{-\frac{1}{8}\sigma_-}\Omega_-.
\end{align}
We conclude this section with a bit of linear algebra that will be needed later.
\begin{lem}\label{14:37:27}
$\text{ker}([I_+,I_-])=\text{ker}(I_++I_-)\oplus\text{ker}(I_+-I_-)$.
\end{lem}
\begin{proof}
Clearly $\text{ker}(I_++I_-)\cap\text{ker}(I_+-I_-)=0$, and using 
\begin{align*}
[I_+,I_-]=(I_+-I_-)(I_++I_-)=-(I_++I_-)(I_+-I_-)
\end{align*}
we see that $\text{ker}([I_+,I_-])\supset \text{ker}(I_++I_-)\oplus\text{ker}(I_+-I_-)$. As both $I_\pm$ anti-commute with $[I_+,I_-]$, they preserve $\text{ker}([I_+,I_-])$. On this subspace $I_+$ and $I_-$ commute with each other and so they admit a simultaneous eigenspace decomposition, all of whose eigenvalues are $\pm i$. The result follows.  
\end{proof}
\noindent In particular, we have $\text{ker}(Q)=\text{ker}(\pi_{\J_1})\oplus \text{ker}(\pi_{\J_2})$

\section{A flow of bi-Hermitian structures}
\label{17:46}

In the process of blowing up generalized K\"ahler manifolds one encounters metrics which degenerate along the exceptional divisor. To deal with this we introduce a deformation procedure to flow a degenerate structure into a non-degenerate one. The idea behind this flow already appears in \cite{MR2371181} where it was used to describe new examples of generalized K\"ahler manifolds, and it was subsequently used in \cite{MR2861778} for the blow-up procedure. The following definition is intended to capture the situation encountered in the blow-up. 
\begin{defn}\label{14:17:27}
A \textsl{degenerate bi-Hermitian structure} on $M$ is a tuple $(g,I_+,I_-,H)$, where $g$ is a symmetric bilinear form, $I_\pm$ are complex structures and $H$ is a closed $3$--form, such that 
\begin{itemize}
\item[i)] $g$ is positive on $M\backslash E$, where $E\subset M$ is a closed and nowhere dense submanifold, on which $TM^\perp$ has constant rank.
\item[ii)] $I_\pm$ are compatible with $g$ and satisfy the integrability condition $\pm d^c_\pm \omega_\pm=H$.
\item[iii)] The bivector $Q:=-\frac{1}{2}[I_+,I_-]g^{-1}$, defined on $M\backslash E$, extends smoothly over $M$.  
\end{itemize}
\end{defn}
\noindent Since the structure is bi-Hermitian on a dense set, it follows that $Qg=-\frac{1}{2}[I_+,I_-]$ holds everywhere and that $\sigma_\pm:=Q-iI_\pm Q$ is holomorphic Poisson with respect to $I_\pm$. To define the flow, we need the following extra ingredient.
\begin{defn}
A \textsl{potential} for $I_+$ (respectively $I_-$) is a closed $1$--form $\alpha$ defined on an open dense set, such that $X_\alpha:=Q(\alpha)$ and $d^c_+\alpha$ (respectively $d^c_-\alpha$) extend smoothly over $M$. 
\end{defn}
\begin{rem}
The terminology originates from the situation where $\alpha=-df$ for a densely defined function $f$, which is usually referred to as the potential. Although this is the situation in which we are interested, we state the results in this section for general closed $1$--forms. 
\end{rem}
\noindent Let $\alpha$ be a potential for $I_+$. Denote by $\varphi_t$ the flow of $X_\alpha$ and define closed $2$-forms  
\begin{align*}
G^\pm_t:=(\varphi_t)_\ast d^c_\pm \alpha, \ \ \  \ \ \  F^\pm_t:=\int_0^t G^\pm_s ds.
\end{align*}
We will have to be careful with $G^-_t$ and $F^-_t$, since $d^c_-\alpha$ is not assumed to be smooth. The aim of this section is to prove
\begin{thm}
\label{14:31:30} 
Let $(g,I_+,I_-,H)$ be a degenerate bi-Hermitian structure with compact degeneracy submanifold $E$, and let $\alpha$ be a potential for $I_+$ such that $d^c_+\alpha$ has compact support and $I_-^\ast(-d^c_+\alpha)^{1,1}_{I_-}$ is positive on $TM^\perp$. 
Then the tuple $(g_t,I_{+,t},I_{-,t},H_t)$, where 
\begin{align}
\label{16:56:07}
 g_t:=g-I_-^\ast \big(F^+_t\big)^{1,1}_{I_-}, \  I_{+,t}:=(\varphi_t)_\ast I_{+},  \  I_{-,t}:=I_-,  \ H_t:=H+id\Big(\big(F^+_t)^{2,0}_{I_-}-\big(F^+_t)^{0,2}_{I_-}\Big),
\end{align}
forms a bi-Hermitian structure for sufficiently small $t> 0$.
\end{thm}
\begin{proof}
\noindent It is clear that $g_t$ is symmetric, $I_{\pm,t}$ are integrable and that $H_t$ is closed for all $t$. By construction $g_t$ is compatible with $I_{-,t}=I_-$. Let us show that $g_t$ is a metric for sufficiently small $t>0$. Choose a relatively compact open neighborhood $V$ of $E$ in $M$ with $\text{supp}(d^c_+\alpha)\subset V$ and pick a $\delta_1>0$ and a relatively compact open set $W$ with $\overline{\varphi_t(V)}\subset W$ for all $t\leq \delta_1$. By construction $F^+_t= 0$ on $M\backslash W$ for $t\leq \delta_1$ and so $g_t=g$ is non-degenerate there. Writing $g_t=g+h_t$, then as $t$ goes to $0$, $h_t/t$ converges to $\dot{h}_0$ which by assumption is positive on $TM^\perp$. For small $\epsilon>0$, $g+\epsilon\dot{h}_0$ is positive on $TM|_E$ and therefore also on $TM|_{\overline{U}}$ for $U\subset W$ a small enough neighborhood of $E$. Hence the same is true for $g+\epsilon h_t/t$, and therefore also for $tg/\epsilon +h_t$, provided $t>0$ is close to zero. If in addition $t\leq \epsilon$ then $tg/\epsilon +h_t\leq g+h_t=g_t$ since $g\geq 0$. In conclusion, there exist a neighborhood $U$ of $E$ in $W$ and a $\delta_2>0$ such that $g_t$ is positive on $U$ for all $0<t\leq \delta_2$. Since $\overline{W}\backslash U$ is compact and $g_0=g$ is non-degenerate on $M\backslash E$ there is a $\delta_3>0$ such that $g_t$ is positive on $\overline{W}\backslash U$ for $0\leq t\leq \delta_3$. Consequently, $g_t$ is a metric for $0<t\leq \text{min}(\delta_1,\delta_2,\delta_3)$. 
\newline
As already observed above, $g_t$ is compatible with $I_{-,t}$ for all $t$. Moreover, since 
any closed $2$--form $F$ on a complex manifold satisfies
\begin{align*}
d^c F^{1,1}=id (F^{2,0}-F^{0,2}),
\end{align*} 
we obtain
\begin{align*}
-d^c_{-,t} \omega_{-,t}=-d^c_-(\omega_--(F^+_t)^{1,1}_{I_-})=H+id\big((F^+_t)^{2,0}_{I_-}-(F^+_t)^{0,2}_{I_-}\big)=H_t.
\end{align*}
Hence all that is left to verify is that $I_{+,t}$ is also compatible with $g_t$ and that $d^c_{+,t}\omega_{+,t}=H_t$. In contrast with the $I_{-,t}$ case this is not immediately obvious, the reason being that the flow seems to treat $I_+$ and $I_-$ on an unequal footing. However, we will now show that if we pull back (\ref{16:56:07}) by $\varphi_t$, we obtain a similar flow but with the roles of $I_+$ and $I_-$ interchanged. We begin by giving an alternative formula for $I_{+,t}$.
\begin{lem} 
\label{13:23:32}
${\varphi_t}_\ast I_{\pm}=I_{\pm}-QF^\pm_t$. 
\end{lem}
\begin{proof}
Consider the generalized complex structure 
\begin{align*}
\J_+=\begin{pmatrix} I_+ & Q \\ 0 & -I_+^\ast \end{pmatrix}, 
\end{align*}
integrable with respect to the zero $3$--form (c.f. Example \ref{11:37:30}). As $\alpha$ is closed, the generalized vector field $\J_+ \alpha=X_\alpha-I^\ast_+ \alpha$ is a symmetry, which means that $\J_+$ is preserved by its flow\footnote{Although $\J_+ \alpha$ is only densely defined, its associated adjoint action on $\Gamma(\T M)$ depends only on $X_\alpha$ and $d^c_+\alpha$ and is therefore defined everywhere. In particular, the flow is also defined everywhere.} 
\begin{align*}
\psi^+_t=e^{F^+_t}_\ast\circ (\varphi_t)_\ast. 
\end{align*}
Hence $e^{-F^+_t}_\ast \circ \J_+ \circ e^{F^+_t}_\ast ={\varphi_t}_\ast \J_+$ and in particular ${\varphi_t}_\ast I_{+}=I_{+}-QF^+_t$. For $I_-$ we have to be careful since we do not know whether $F^-_t$ is smooth. We can apply the above argument on the open dense set where $\alpha$ is smooth, if we keep the time parameter small. In particular, we conclude that $\mcL_{X_\alpha}I_-=Q d^c_-\alpha$ holds on the dense set where $\alpha$ is defined. Since $X_\alpha$ is smooth, we learn that $Qd^c_-\alpha$, and therefore also $QG^-_t$ and $QF^-_t$, are smooth. The statement for $I_-$ then follows because it holds at $t=0$, and because both sides have equal time derivatives.
\end{proof}
\noindent In the lemma below we denote by $\nabla$ the Levi-Cevita connection and by $\nabla^\pm$ the connections defined in Proposition \ref{11:28:39}, which are defined on the open set $M\backslash E$ where $g$ is a metric. 
\begin{lem}\label{13:24:30} 
\begin{align}
\label{18:00:36}
\mcL_{X_\alpha}g=&I_-^\ast (d^c_+\alpha)^{1,1}_{I_-}-I_+^\ast (d^c_-\alpha)^{1,1}_{I_+} \\
\iota_{X_\alpha}H=&\frac{1}{2}\big(d^c_+(I_-^\ast\alpha)+d_-^c(I_+^\ast\alpha)\big)
\label{18:21:36}
\end{align} 
\end{lem}
\begin{proof} We will verify these expressions on the intersection of $M\backslash E$ with the open dense set where $\alpha$ is defined. Since the left-hand sides of both equations are smooth, this shows that the right-hand sides have smooth extensions over all of $M$. 

\noindent As both sides of (\ref{18:00:36}) are symmetric, it suffices to evaluate them on a pair $(Y,Y)$. Using $\mcL_{g^{-1}\alpha}g(Y,Y)=2\nabla_Y\alpha(Y)$ and $X_\alpha=g^{-1}(\frac{1}{2}[I_+,I_-]^\ast \alpha)$, we obtain 
\begin{align*} 
\mcL_{X_\alpha}g(Y,Y)=&\nabla_Y\big([I_+,I_-]^\ast \alpha\big)(Y)\\
=&\nabla_Y\alpha\big([I_+,I_-]Y\big)+\alpha\big((\nabla_Y[I_+,I_-])Y\big).
\end{align*}
Using the identities\footnote{Here $I_\pm^\ast$ denotes the action on forms which on a form of degree $(p,q)$ acts by $i(p-q)$.} $d^c_\pm \beta=[d,I_\pm^\ast]\beta$ and $d\beta=(\nabla\beta)^{sk}$ for $\beta\in \Omega^\ast(M)$, together with the fact that $\alpha$ is closed, the right hand side of (\ref{18:00:36}) evaluates to\footnote{Here we use the notation $(+\leftrightarrow -)$ to denote the same term that precedes it but with $\pm$ symbols interchanged.} 
\begin{align}\label{09:17:48}
d(I_+^\ast\alpha)(Y,I_-Y)-(+\leftrightarrow-)=&\nabla_Y\alpha(I_+I_-Y)+\alpha((\nabla_YI_+)I_-Y)-\nabla_{I_-Y}\alpha(I_+Y)\nonumber\\&-\alpha((\nabla_{I_-Y}I_+)Y)-(+\leftrightarrow-)\nonumber \\  
=&\nabla_Y\alpha([I_+,I_-]Y)+\alpha\big((\nabla_Y[I_+,I_-])Y\big)\nonumber\\&+\alpha\Big(I_-(\nabla_YI_+)Y -(\nabla_{I_-Y}I_+)Y-(+\leftrightarrow -)\Big).
\end{align}
Using $\nabla_Y I_\pm=\pm\frac{1}{2}[g^{-1}\iota_YH,I_\pm]$, a tedious but straightforward calculation shows that 
\begin{align}\label{08:54:46}
(I_\mp \nabla_YI_\pm-\nabla_{I_\mp Y}I_\pm)Z=\mp\frac{1}{2}g^{-1}\big(I_\mp^\ast \iota_{{}_{I_\pm Z}}\iota_{{}_Y}+I_\pm^\ast\iota_{{}_Z}\iota_{{}_{I_\mp Y}}+I_\mp^\ast I_\pm^\ast\iota_{{}_Z}\iota_{{}_Y}+\iota_{{}_{I_\pm Z}}\iota_{{}_{I_\mp Y}}	\big)H.
\end{align}
From this we see that the last term in (\ref{09:17:48}) vanishes, proving (\ref{18:00:36}). For (\ref{18:21:36}), we compute
\begin{align*}
d^c_+(I_-^\ast \alpha)(Y,Z)=&d(I_+^\ast I_-^\ast\alpha)(Y,Z)-d(I_-^\ast\alpha)(I_+Y,Z)-d(I_-^\ast\alpha)(Y,I_+Z)\\ =&-\nabla_{I_+Y}\alpha(I_-Z)+\alpha\big(I_-(\nabla_YI_+)Z-(\nabla_{I_+Y}I_-)Z\big)-(Y\leftrightarrow Z).
\end{align*}
Hence, using again (\ref{08:54:46}) and the fact that $\nabla\alpha$ is symmetric, we get  
\begin{align*}
\big(d^c_+(I_-^\ast\alpha)+d_-^c(I_+^\ast\alpha)\big)(Y,Z)=&\alpha\big(I_-(\nabla_YI_+)Z-(\nabla_{I_-Y}I_+)Z+(+\leftrightarrow -)	\big)-(Y\leftrightarrow Z)\\
=&\alpha\big(g^{-1}(I^\ast_+I^\ast_--I^\ast_-I^\ast_+)\iota_Z\iota_YH	\big)\\
=&2(\iota_{X_\alpha}H)(Y,Z),
\end{align*}
proving (\ref{18:21:36}). 
\end{proof}
\noindent From the proof of Lemma \ref{13:23:32} we learned that $Qd^c_-\alpha$ is smooth, and therefore also ${d^c_-\alpha} Q$ by taking adjoints. Equation (\ref{18:00:36}) gives us in addition smoothness of $(d^c_-\alpha)^{1,1}_{I_+}$, and if we apply $d$ to Equation (\ref{18:21:36}) we see (after some rearranging) that $d(I_+^\ast d^c_-\alpha+d^c_-\alpha I_+)$ is also smooth. Combining this with Lemma \ref{13:23:32}, we conclude that both $I_+^\ast F^-_t-F^-_tI_+$ and $d(I_+^\ast F^-_t+F^-_t I_+)$ are smooth, which is what we need to make sense of the following lemma.
\begin{lem}
\label{17:07:45}
Let $(g_t,I_{+,t},I_-,H_t)$ be as in (\ref{16:56:07}). Then 
\begin{align}
\label{17:19:05}
\varphi^\ast_tg_t&=g+\frac{1}{2}\big(I_+^\ast F^-_{-t}-F^-_{-t} I_+\big),
\\ 
\label{12:22:59}
\varphi_t^\ast H_t&=H+\frac{1}{2}d\big(I_+^\ast F^-_{-t} + F^-_{-t} I_+\big).
\end{align}
\end{lem}
\begin{proof}
Both equations hold at $t=0$, so it suffices to show that both sides have the same time derivative. Since $\varphi_t^\ast F_t^\pm=-F^\pm_{-t}$ we have $\varphi_t^\ast g_t=\varphi_t^\ast g+\frac{1}{2}(I_{-,-t}^\ast F^+_{-t}-F^+_{-t}I_{-,-t})$. Using Lemma \ref{13:23:32} and Equation (\ref{18:00:36}) we obtain
\begin{align*}
\frac{d}{dt} (\varphi^\ast_tg_t)=&\varphi_t^\ast(\mcL_{X_\alpha}g)+\frac{1}{2}\big(G^-_{-t}QF^+_{-t}-I_{-,-t}^\ast G^+_{-t}-F^+_{-t}QG^-_{-t}+G^+_{-t}I_{-,-t}\big) \\ 
=&\varphi_t^\ast\Big(	\mcL_{X_\alpha}g+\frac{1}{2}\big( -d^c_-\alpha QF^+_t-I_-^\ast d^c_+\alpha+F^+_tQd^c_-\alpha+d^c_+\alpha I_-				\big)	\Big) \\
=&\frac{1}{2}\varphi_t^\ast\big(-I_{+,t}^\ast d^c_-\alpha +d^c_-\alpha I_{+,t}	\big)\\
=&\frac{1}{2}(-I^\ast_+G^-_{-t}+G^-_{-t}I_+). 
\end{align*}
This equals the time derivative of the right-hand side, thereby proving (\ref{17:19:05}).
For (\ref{12:22:59}), using $\varphi_t^\ast H_t=\varphi_t^\ast H-\frac{1}{2}d(I_{-,-t}^\ast F^+_{-t}+F^+_{-t}I_{-,-t})$, we compute
\begin{align*}
\frac{d}{dt}(\varphi_t^\ast H_t)&=\varphi_t^\ast(\mcL_{X_\alpha}H)-\frac{1}{2}d\big( G^-_{-t}QF^+_{-t} -I_{-,-t}^\ast G^+_{-t}+F^+_{-t}QG^-_{-t}-G^+_{-t}I_{-,-t}	\big)\\
=&\frac{1}{2}\varphi_t^\ast d\big(2\iota_{X_\alpha}H+d^c_-\alpha QF^+_t+I_-^\ast d^c_+\alpha+F^+_tQd^c_-\alpha+d^c_+\alpha I_-	\big).
\end{align*}
We can rewrite (\ref{12:22:59}) as 
\begin{align*}
\iota_{X_\alpha}H=\frac{1}{2}d\big((I^\ast_+I^\ast_-+I^\ast_-I^\ast_+)\alpha\big)-\frac{1}{2}(I_+^\ast d^c_-\alpha+d^c_-\alpha I_++I_-^\ast d^c_+\alpha+d^c_+\alpha I_-),
\end{align*}
and since the first term is closed, we obtain
\begin{align*}
\frac{d}{dt}(\varphi_t^\ast H_t)=\frac{1}{2}\varphi_t^\ast d\big( -I_{+,t}^\ast d^c_-\alpha-d^c_-\alpha I_{+,t}		\big)=-\frac{1}{2}d(I_+^\ast G^-_{-t}+G^-_{-t}I_+),
\end{align*}
which equals the time derivative of the right-hand side of (\ref{12:22:59}), thereby proving it. 
\end{proof}
\noindent From Lemma \ref{17:07:45}, together with the arguments applied before to $I_{-,t}$, it follows that $\varphi_t^\ast I_{+,t}=I_+$ is compatible with $\varphi_t^\ast g_t$, and that 
\begin{align*}
\varphi_t^\ast (d^c_{+,t}\omega_{+,t})=d^c_+((\varphi_t^\ast g_t)I_+)=d^c_+(\omega_++(F^-_{-t})^{1,1}_{I_+})=H+id\big((F^-_{-t})^{2,0}_{I_+}-(F^-_{-t})^{0,2}_{I_+}\big)=\varphi_t^\ast H_t.
\end{align*}
Pushing everything forward again by $\varphi_t$ we obtain the desired compatibility of $I_{+,t}$ with $g_t$ and $H_t$, finishing the proof of Theorem \ref{14:31:30}.
\end{proof}
\begin{rem}
We stated the theorem for potentials for $I_+$ but of course a similar result is true for potentials for $I_-$. In that case we need $I_+^\ast(-d^c_-\alpha)^{1,1}_{I_+}$ to be positive on $TM^\perp$.
\end{rem}
\begin{rem}[\cite{MR2371181}]
\label{09:29:21}
Theorem \ref{14:31:30} is stated for metrics which are almost everywhere non-degenerate. It can however, also be applied to the following situation where $g$ is identically zero. Suppose $(M,I)$ is a compact complex manifold and $\sigma$ a holomorphic Poisson structure with real part $Q$. Setting $g=0$, $I_\pm=I$ and $H=0$, this is a degenerate bi-Hermitian structure in the sense of Definition \ref{14:17:27}, except for the fact that the degeneracy set $E=M$ is no longer nowhere dense. Still, Lemma \ref{13:23:32} only used that $Q$ is the real part of holomorphic Poisson structures $\sigma_\pm$ while Lemma \ref{13:24:30} is trivially true in this case (both sides of both equations are zero). Hence Theorem \ref{14:31:30} still applies in this case, and all we need is a potential $\alpha$ such that $-d^c\alpha$ is positive on $M$. To that end, suppose that $D\subset M$ is a divisor which is Poisson for $\sigma$ and which in addition is positive, i.e.\ the line bundle $\mathcal{O}_X(D)$ is positive. Let $s\in \Gamma(\mathcal{O}_X(D))$ be a holomorphic section which vanishes to first order along $D$ and choose a Hermitian metric $h$ on $\mathcal{O}_X(D)$ such that $iR_h$ is positive, where $R_h$ is the curvature of the unitary connection induced by $h$. Then $\alpha:=-d\log |s|$ is a potential with the desired properties, for $-d^c\alpha=iR_h$ is smooth and positive, while $Q(\alpha)$ is smooth because $D$ is Poisson (c.f. the proof of Theorem \ref{10:31}($i)$). The deformation procedure then gives us a bi-Hermitian structure where $I_+$ and $I_-$ are no longer equal. In \cite{MR2371181} this was applied to find examples of generalized K\"ahler structures on Del Pezzo surfaces, which are complex surfaces whose anticanonical bundle is positive. 
\end{rem}

\section{Blowing up submanifolds}
\label{09:19:55}

\subsection{Blow-ups in generalized complex geometry}
\label{11:06:22}

We briefly summarize the results on blowing up submanifolds in generalized complex geometry as treated in \cite{GCblowups}. The most general notion of blowing up a submanifold $Y\subset M$ involves the notion of \textsl{holomorphic ideal}; an ideal $I_Y\subset C^\infty(M)$ of complex valued smooth functions with $Y$ as its zero set, and which is locally around a point in $Y$ generated by functions $z^1,\ldots, z^l$ that form a submersion to $\C^l$. The \textsl{blow-up} $\pi:\tilde{M}\rightarrow M$ of $Y$ in $M$ is then defined by the following universal property: for any map $f:X\rightarrow M$ for which $f^\ast I_Y$ is a divisor\footnote{By definition this is an ideal with nowhere dense zero set and which locally is generated by a single function.} there is a unique factorization of $f$ through $\pi$. It is shown in \cite{GCblowups} that for every holomorphic ideal there exists a (canonical) blow-up. 

There are two kinds of generalized complex submanifolds\footnote{See Definition \ref{20:22} for the notion of generalized complex submanifold. Also, by convention all our submanifolds are closed in $M$.} that admit a blow-up. Firstly, we have the \textsl{generalized Poisson submanifolds}, by definition these are the submanifolds $Y\subset M$ with $\J N^\ast Y=N^\ast Y$, and they behave in a complex manner in normal directions. It turns out that these submanifolds inherit a canonical holomorphic ideal from the generalized complex structure, and so there is a canonical (differentiable) blow-up. This is however not automatically generalized complex. To phrase the precise condition, observe that a generalized Poisson submanifold is in particular Poisson for $\pi_\J$, and so $N^\ast Y$ inherits a fiberwise Lie algebra structure. Concretely, if $\alpha,\beta\in N^\ast_yY$ and $X\in T_yM$, we have
\begin{align*}
[\alpha,\beta] (X)=(\mcL_{\widetilde{X}}\pi_\J)(\alpha,\beta),
\end{align*}
where $\widetilde{X}$ is an arbitrary local extension of $X$. This Lie bracket turns out to be complex linear, where the complex structure on $N^\ast Y$ is given by the restriction of $\J$. We call $N^\ast Y$ \textsl{degenerate} if the bundle map
\begin{align*}
 \Lambda^3 N^\ast Y&\rightarrow Sym^2(N^\ast Y) \\ u\wedge v\wedge w&\mapsto u[v,w]+v[w,u]+w[u,v]
\end{align*} 
vanishes. Here both the exterior and symmetric algebra are taken over $\C$, which is well-defined since $N^\ast Y$ is complex. Note that degeneracy is really a property of Lie algebras, and $N^\ast Y$ being degenerate means that all its fibers are. Another description of degeneracy for a Lie algebra is that the bracket of any two elements lies in the plane spanned by them. Then, the differentiable blow-up of $Y$ in $M$ admits a generalized complex structure for which the blow-down map is holomorphic if and only if $N^\ast Y$ is degenerate. 
\newline
\newline
The second class of submanifolds is formed by the \textsl{generalized Poisson transversals}, defined by the condition $\J N^\ast Y\cap (N^\ast Y)^\perp =0$. They behave symplectically in normal directions and admit a global neighborhood theorem.  If in addition $Y$ is compact, using the neighborhood theorem one can construct a (non-canonical) generalized complex blow-up. To put this blow-up in the context of holomorphic ideals, observe that a submanifold $Y$ equipped with a complex structure on its normal bundle inherits many holomorphic ideals by taking the canonical ideal of fiberwise linear functions on the manifold $NY$, and transporting it into $M$ using a tubular neighborhood. Then the blow-up of a generalized Poisson transversal can be regarded as the blow-up with respect to such an ideal. Since it is non-canonical, the ideal description is not really useful in this context.

\noindent On a generalized K\"ahler manifold the two types of submanifolds are related as follows.
\begin{lem}
Let $(M,\J_1,\J_2)$ be a generalized K\"ahler manifold. A submanifold $Y\subset M$ which is generalized Poisson for $\J_1$ is a generalized Poisson transversal for $\J_2$. 
\end{lem}
\begin{proof}
By the generalized K\"ahler condition we have 
\begin{align*}
\langle \J_1\alpha,\J_2\alpha\rangle >0 \ \ \ \forall \alpha\in N^\ast Y.
\end{align*}
So if $\J_1N^\ast Y=N^\ast Y$, then obviously $\J_2N^\ast Y\cap (N^\ast Y)^\perp=0$. 
\end{proof}
\noindent
In particular, if $Y$ is a generalized Poisson submanifold for $\J_1$ such that $N^\ast Y$ is degenerate, then $Y$ can be blown up for both $\J_1$ and $\J_2$, albeit in different manners. It sounds reasonable to expect that there will be a generalized K\"ahler blow-up for $Y$ in $M$. Although a complete answer is still lacking, in the next sections we will give some extra sufficient conditions that will guarantee the existence of a generalized K\"ahler blow-up. 

The strategy will be the following: Given a $Y\subset (M,\J_1,\J_2)$ as above, we blow it up for $\J_1$ and then show that the bi-Hermitian structure lifts to a degenerate bi-Hermitian structure on the blow-up. Then, under additional hypotheses we apply the flow procedure of the previous section which makes the structure non-degenerate, and the result will be the desired blow-up.  

\subsection{Lifting the bi-Hermitian structure}

Let $Y\subset (M,\J_1,\J_2)$ be a generalized Poisson transversal for $\J_1$. Associated to $(\J_1,\J_2)$ there is the pair of holomorphic Poisson structures $(I_\pm, \sigma_\pm)$ defined in (\ref{15:18:13}), whose real parts equal $Q$. We want to show that $(I_\pm, \sigma_\pm)$ all lift to the generalized complex blow-up of $Y$ with respect to $\J_1$. To that end we will first prove that $\sigma_\pm$ lifts to the blow-up of $Y$ with respect to $I_\pm$, and then show that the three different blow-ups coincide.  

From (\ref{10:42:52}) it follows that $Y$ is generalized Poisson for $\J_1$ if and only if $I_\pm^\ast N^\ast Y=N^\ast Y$ and $I^\ast_+|_{N^\ast Y}=I^\ast_-|_{N^\ast Y}$. Consequently, $Y$ is a complex Poisson submanifold for both $(I_\pm,\sigma_\pm)$.  

\begin{lem}
\label{14:25:15}
$N^\ast Y$ is degenerate for $\pi_{\J_1}$ if and only if it is degenerate for $Q$.
\end{lem}
\begin{proof}
From Equations (\ref{09:06}) and (\ref{09:05}) we obtain 
\begin{align}
\label{18:11:34}
Q=\pi_{\J_1}\circ (I_++I_-)^\ast.
\end{align}
The endomorphism $A:=(I_++I_-)^\ast$ restricts to an automorphism of $N^\ast Y$ and we have
\begin{align}
[\alpha,\beta]_Q=d_y\big( Q(\tilde{\alpha},\tilde{\beta})\big)=d_y\big( \pi_{\J_1}(A\tilde{\alpha},\tilde{\beta})\big)= [A\alpha,\beta]_{\pi_{\J_1}}
\label{18:01}
\end{align}
for $\alpha,\beta\in N_y^\ast Y$, where $\tilde{\alpha},\tilde{\beta}$ are smooth local extensions of $\alpha$ and $\beta$. Abstractly, if $(\mfg, [,])$ is a Lie algebra and $A:\mfg \rightarrow \mfg$ a linear map such that $[u,v]_A:=[Au,v]$ is again a Lie bracket\footnote{In fact as the argument shows, the Jacobi identity plays no role here. Hence this is really a statement about skew-symmetric brackets.}, then degeneracy of $[,]$ implies that for $[,]_A$ as well. To prove this, we need to show that $[x,y]_A\in \C\cdot x +\C\cdot y$ for all $x,y\in \mfg$. It suffices to verify this for $x,y$ that are linearly independent. Since $[,]$ is degenerate, there are $\lambda,\mu\in \C$ with $[x,y]_A=[Ax,y]=\lambda Ax+\mu y$. Since $[,]_A$ is skew, we have $[Ax,y]=[Ax,x+y]$, and so there are $\lambda',\mu'\in \C$ with $[x,y]_A=\lambda' Ax+\mu' (x+y)$. Comparing both equations, we see that either $Ax\in \C\cdot x+\C\cdot y$ and hence also $[Ax,y]\in \C\cdot x+ \C\cdot y$,  or $[Ax,y]=\lambda Ax$. Running the same argument with $x$ and $y$ interchanged we see that either $[Ax,y]\in \C\cdot x+\C\cdot y$, or $[Ax,y]=\mu Ay$ for some $\mu\in \C$. If $[Ax,y]$ is nonzero, $Ax$ is proportional to $Ay$ and so $[Ax,y]=0$ by skew symmetry of $[,]_A$, a contradiction. Hence, $[x,y]_A\in \C\cdot x +\C\cdot y$ $\forall x,y\in \mfg$, and so $[,]_A$ is degenerate. \end{proof}
\noindent By a result of Polishchuk \cite{MR1465521} it follows that $Y$ can be blown up for both $(I_\pm,\sigma_\pm)$. Let us denote by $\tilde{M}$ the blow-up with respect to $\J_1$ and by $\tilde{M}_\pm$ the blow-up with respect to $I_\pm$.  
\begin{lem}
The blow-ups $\tilde{M}$, $\tilde{M}_+$ and $\tilde{M}_-$ all coincide.
\end{lem}
\begin{proof}
As discussed in Section \ref{11:06:22} and more precisely explained in \cite{GCblowups}, the blow-up $\tilde{M}$ is constructed from a holomorphic ideal $I_{Y,\J_1}$ that $Y$ inherits from $\J_1$, while the blow-ups $\tilde{M}_\pm$ use the natural holomorphic ideals $I_{Y,I_\pm}$  that $Y$ inherits from being a complex submanifold of $(M,I_\pm)$. It thus suffices to show that these three ideals coincide, which turns out to be true up to a conjugation, i.e.\ $\overline{I_{Y,\J_1}}=I_{Y,I_\pm}$. This is not a problem, for the blow-up of a conjugate ideal is given by the same manifold but with conjugate divisor. Pick a local chart $(\R^{2n-2k},\omega_{st}) \times (\C^k,\sigma)$ for $\J_1$ as provided by Theorem \ref{16:12:53}, in which $Y$ necessarily looks like $W\times Z$ where $W\subset \R^{2n-2k}$ is open and $Z\subset \C^k$ is holomorphic Poisson. If $(x,z)$ are coordinates on this chart in which $Y=\{z^1,\ldots,z^l=0\}$, it is shown in \cite{GCblowups} that the ideal $\langle z^1,\ldots,z^l \rangle$ is independent of the choice of local chart, so they glue together to form (by definition) the ideal $I_{Y,\J_1}$. Let us verify that $\overline{I_{Y,\J_1}}=I_{Y,I_+}$, the case of $I_-$ being similar. Pick a holomorphic chart for $I_+$ with coordinates $u^i$ so that $Y$ is given by $\{u^1,\ldots, u^l=0\}$ and so $I_{Y,I_+}=\langle u^1,\ldots,u^l	\rangle$. As is proven for instance in \cite{MR0212575}, we can verify $ \overline{I_{Y,\J_1}}\subset I_{Y,I_+}$ merely by looking at Taylor series, i.e.\ we need to show that 
\begin{align}
\left. \frac{\partial^m \bar{z}^i}{\partial \bar{u}^{i_1}\ldots \partial \bar{u}^{i_m}}\right |_Y=0 \ \ \ \forall m\geq 0, \ \forall i, i_1,\ldots, i_m\in\{ 1,\ldots, l\}.
\label{10:17} 
\end{align}
For this we use (\ref{09:29}), which now explicitly becomes
\begin{align}
e^{-i\omega_{st}}e^{-\bar{\sigma}}(d\bar{z}^1\ldots d\bar{z}^k)	\wedge \bar{\rho}_2=e^fe^{-\frac{1}{8}\sigma_+}(du^1\ldots du^n),
\label{11:07}
\end{align}
where $e^f$ is some rescaling. One can prove (\ref{10:17}) by induction on $m$ and by applying appropriate Lie derivatives to (\ref{11:07}). For precise details on this part of the argument we refer to the proof of Proposition 2.6 in \cite{GCblowups}, which is almost identical. Hence $ \overline{I_{Y,\J_1}}\subset I_{Y,I_+}$, and since both are holomorphic ideals for $Y$, we obtain $ \overline{I_{Y,\J_1}}= I_{Y,I_+}$, which is what we needed to show.\end{proof}

\subsection{Flowing towards a non-degenerate structure}

Let $Y\subset (M,\J_1,\J_2)$ be a generalized Poisson transversal for $\J_1$ with degenerate normal bundle, and denote by $\pi:\tilde{M}\rightarrow M$ the corresponding blow-up. By Lemma \ref{14:25:15} we know that the complex structures $I_\pm$ lift to $\tilde{M}$, and together with $\pi^\ast g$, $\pi^\ast H$ and the lift of $Q$ they form a degenerate bi-Hermitian structure on $\tilde{M}$. The metric degenerates along the exceptional divisor $E=\pi^{-1}(Y)$, and $T\tilde{M}^\perp$ equals the vertical tangent bundle of the fibration $\pi:E\cong\mbP(NY)\rightarrow Y$. In order to apply the deformation procedure from Section \ref{17:46} we need a suitable potential\footnote{We will consider potentials $\alpha=df$ and also refer to $f$ as the potential, slightly abusing terminology from Section \ref{17:46}.} $f$. This will be based on the following idea (see also \cite{MR2861778}). Consider $M$ and $\tilde{M}$ as complex manifolds with respect to either $I_+$ or $I_-$, and $E$ as a divisor on $\tilde{M}$ with associated holomorphic line bundle $\mathcal{O}_{\tilde{M}}(-E)$. Recall that a Hermitian metric on a holomorphic line bundle induces a unitary connection, whose curvature $R_h$ is of type $(1,1)$. 
\begin{lem}
\label{11:45:57}
If $U$ is any neighborhood of $E$ in $\tilde{M}$, there exists a metric $h$ on $\mathcal{O}_{\tilde{M}}(-E)$ such that $iR_h$ is supported in $U$ and restricts to a positive $(1,1)$--form on $T\tilde{M}^\perp$.
\end{lem}
\begin{proof}
On $E$ we have the tautological line bundle $\mathcal{O}_E(-1)\subset \pi^\ast NY$ whose fiber over a point $l\in E=\mbP(NY)$ is the corresponding line in $NY$. If we equip $NY$ with a Hermitian metric then this induces one on $\mathcal{O}_E(-1)$ and therefore also on $\mathcal{O}_E(1):=\mathcal{O}_E(-1)^\ast$. Denote the latter by $h'$ and its curvature by $R_{h'}$. If we set $E_y:=\pi^{-1}(y)=\mbP(N_yY)$ for $y\in Y$, then $iR_{h'}|_{E_y}$ equals (a multiple of) the Fubini-Study form\footnote{In fact this is one of the standard ways to define the Fubini-Study metric on projective space.} on $\mbP(N_yY)$. Now $\mathcal{O}_{\tilde{M}}(-E)|_{E}\cong N^\ast E$, the conormal bundle of $E$ in $\tilde{M}$, which in turn equals $\mathcal{O}_E(1)$. We can extend the metric $h'$ on $\mathcal{O}_E(1)$ to a metric on $\mathcal{O}_{\tilde{M}}(-E)$ as follows. Forgetting about the holomorphic structure for a moment, pick a tubular neighborhood $p:V\rightarrow E$ with $\overline{V}\subset U$, and use it to identify $\mathcal{O}_{\tilde{M}}(-E)|_V\cong p^\ast \mathcal{O}_E(1)$. Equip $\mathcal{O}_{\tilde{M}}(-E)|_V$ with the metric $p^\ast h'$ with curvature $p^\ast R_{h'}$, which has the same restriction to all the $E_y$'s as $R_{h'}$ does. On the complement of $E$ the bundle $\mathcal{O}_{\tilde{M}}(-E)$ is trivial so can be given a flat metric $h''$. We let $h$ be equal to $h''$ on $\tilde{M}\backslash U$, $p^\ast h'$ on a neighborhood of $E$ in $V$ and a suitable interpolation in between. Clearly $R_h$ is compactly supported in $U$ and $iR_h|_{E_y}$ is positive. 
\end{proof}
\noindent It is a well-known fact that if $s$ is any meromorphic section of $\mathcal{O}_{\tilde{M}}(-E)$ which is not identically zero, then $iR_h=-dd^c\log|s|$. So $f:=-\log |s|$ is a good candidate for a potential. Unfortunately, we can not always guarantee smoothness of the Hamiltonian vector field $Q(df)$. The theorem below gives two situations where $Q(df)$ can be controlled. 
\begin{thm}\label{10:31}
Let $(M,\J_1,\J_2)$ be a generalized K\"ahler manifold and $Y$ a compact generalized Poisson submanifold for $\J_1$ whose conormal bundle is degenerate. Then if one of the following conditions holds, the blow-up $\tilde{M}$ with respect to $\J_1$ has a generalized K\"ahler structure.
\begin{itemize}
\item[i)] ($N^\ast Y, [ , ]_{\pi_{\J_1}})$ is Abelian.    
\item[ii)] $Y\subset D$, where  $D$ is a compact Poisson divisor in $M$ with respect to $(I_+,\sigma_+)$ or $(I_-,\sigma_-)$. 
\end{itemize}  
Moreover, in situation $i)$ the generalized K\"ahler structure on the blow-up agrees with the original structure on the complement of a neighborhood of the exceptional divisor. In situation $ii)$ the same is true, if we make the additional assumption that $\mathcal{O}_M(D)|_Y$ is trivial. In that case it is also not necessary to assume that $D$ is compact.
\end{thm}  
\begin{proof} $i)$: Consider $M$ as a complex manifold with respect to, say, $I_+$. From (\ref{18:01}) we see that $[,]_Q$ is Abelian on $N^\ast Y$ and therefore $E\subset \tilde{M}$ is a Poisson submanifold for the lift of $Q$. Consider the potential $f=-\log|s|$, where $s$ is a meromorphic section of $\mathcal{O}_{\tilde{M}}(-E)$ with a simple pole along $E$, and the norm is taken with respect to a metric as in Lemma \ref{11:45:57}. We claim that $Q(df)$ extends smoothly to the whole of $\tilde{M}$. To see this, let $x\in E$ and let $e$ be a local holomorphic section of $\mathcal{O}_{\tilde{M}}(-E)$ with $e(x)\neq 0$. Then $s=\frac{1}{z}e$, where $z$ is a local equation for $E$, hence 
\begin{align*}
Q(df)=\frac{1}{4}\sigma_+\big(\frac{dz}{z}\big)+\frac{1}{4}\bar{\sigma}_+\big(\frac{d\bar{z}}{\bar{z}}\big)-Q(d\log|e|).
\end{align*}
Since $\sigma_+(dz)$ vanishes on $E$, it is divisible by $z$ and we see that $Q(df)$ is indeed smooth. 
We already know that $dd^c_+f$ is smooth and that $dd^c_+ f|_{T\tilde{M}^\perp}$ is positive, but in order to apply Theorem \ref{14:31:30} we need that $(I_-^\ast (dd^c_+f)^{1,1}_{I_-})|_{T\tilde{M}^\perp}$ is positive. However, the complex structure on $E_y$ is induced from $N_yY$ under the isomorphism $E_y=\mbP(N_yY)$. Since $Y$ is generalized Poisson, both $I_+$ and $I_-$ coincide on $NY$ and preserve it. So $E_y$ is a complex submanifold of $\tilde{M}$ with respect to both $I_+$ and $I_-$, with the same induced complex structure. In particular $(I_-^\ast (dd^c_+f)^{1,1}_{I_-})|_{E_y}=dd^c_+f|_{E_y}$ is positive. So Theorem \ref{14:31:30} applies and we obtain a generalized K\"ahler structure by perturbing the structure in a neighborhood of $E$, whose size is controlled by the choice of metric in Lemma \ref{11:45:57} (so in particular can be arbitrarily small). 
\newline
$ii)$: Let $\tilde{D}$ denote the proper transform\footnote{By definition this is the closure of $\pi^{-1}(D)\backslash E$ in $\tilde{M}$.} of $D$ on the blow-up $\tilde{M}$. In terms of divisors, $\tilde{D}=\pi^\ast D-kE$ for some $k\in \mathbb{Z}_{>0}$ and so $\mathcal{O}_{\tilde{M}}(\tilde{D})=\mathcal{O}_{\tilde{M}}(-kE)\otimes \pi^\ast \mathcal{O}_M(D)$. Equip $\mathcal{O}_{\tilde{M}}(-kE)=\mathcal{O}_{\tilde{M}}(-E)^{\otimes k}$ with the metric $h^{\otimes k}$, where $h$ is a metric on $\mathcal{O}_{\tilde{M}}(-E)$ as in Lemma \ref{11:45:57}. 
If $h'$ is any metric on $\mathcal{O}_M(D)$, the metric $h^{\otimes k}\otimes \pi^\ast h'$ on $\mathcal{O}_{\tilde{M}}(\tilde{D})$ satisfies $iR_{h^{\otimes k}\otimes \pi^\ast h'}=ikR_h+i\pi^\ast R_{h'}$, which is positive  on $T\tilde{M}^\perp$ since $\pi^\ast R_{h'}$ vanishes there. Let $s$ be a holomorphic section of $\mathcal{O}_{\tilde{M}}(\tilde{D})$ with a simple zero along $\tilde{D}$, and define $f:=-\log |s|$. Then $dd^cf=ikR_h+i\pi^\ast R_{h'}$ is smooth on $\tilde{M}$ and positive on $T\tilde{M}^\perp$, while the same argument as in $i)$ shows that $Q(df)$ is smooth, using the fact that $\tilde{D}$ is Poisson. So again Theorem \ref{14:31:30} applies, but this time the structure is perturbed along $\tilde{D}$ as well so we can not contain the deformation to a neighborhood of $E$. If however we know that $\mathcal{O}_M(D)|_Y$ is trivial, then we can choose $h'$ above to be flat around $Y$ and so $dd^c_+f=ikR_h$ around $E$. If $s'$ is a section of $\mathcal{O}_{\tilde{M}}(kE)$ with a zero of order $k$ along $E$ and $\mathcal{O}_{\tilde{M}}(kE)$ is equipped with the metric dual to $h^{\otimes k}$, we can define $f':=-\rho \cdot \log|s'|$, where $\rho$ is a function which is $0$ near $E$ and $1$ outside of a neighborhood of $E$. Then $f+f'$ still has the property that $Q(df)$ is smooth but in addition satisfies $dd^c(f+f')=0$ on an annulus around $E$. We then apply the deformation procedure only on a neighborhood of $E$, keeping it fixed on an annulus around it, and then glue the result back to the original structure. 
\end{proof}
\begin{rem}\label{13:56:55}
$i)$: Suppose that $M$ is a K\"ahler manifold, seen as a generalized K\"ahler manifold as in Example \ref{10:47:25}, and $Y$ is a complex submanifold regarded as a generalized Poisson submanifold for $\J_1$. Then, since $\pi_{\J_1}=0$, $N^\ast Y$ is Abelian and we are in situation $i)$ of the Theorem. Equation (\ref{16:56:07}) that defines the flow reduces in this equation to simply adding $dd^cf$ to the symplectic form, and this is how one usually produces a K\"ahler metric on the blow-up.
\newline
$ii)$: Let us clarify why we need $\mathcal{O}_M(D)|_Y$ to be trivial if we want to contain the deformation to a neighborhood of $E$ in situation $ii)$. In the first part of the proof we are flowing the structure by the $2$--form $ikR_h+i\pi^\ast R_{h'}$, and in the second part we want to cancel this on an annulus around $E$ by $dd^cf'$, where $f'$ is a smooth function. In particular we need $ikR_h+i\pi^\ast R_{h'}$ to be exact on the annulus, which is automatic for $ikR_h$ since $\mathcal{O}_{\tilde{M}}(kE)|_{\tilde{M}\backslash E}$ is trivial. For $\pi^\ast R_{h'}$ to be exact, we need $R_{h'}$ to be exact\footnote{As $Y$ has complex codimension $2$ or bigger (otherwise the blow-up is trivial), the Gysin sequence shows that the second degree cohomology of an annulus around $Y$ agrees with that of $Y$ itself.} around $Y$, which amounts to $\mathcal{O}_M(D)|_Y$ being trivial around $Y$.

\end{rem}

\noindent Although condition $ii)$ of the theorem is clear as it is stated, it is unclear whether it has any applications. For that reason we state the following
\begin{cor}
Let $(M,\J_1,\J_2)$ be generalized K\"ahler with $\J_1$ generically of symplectic type and $Y$ a compact generalized Poisson submanifold for $\J_1$ with degenerate normal bundle and which is contained in the type change locus\footnote{Since $\J_1$ is generically symplectic, for $Y$ to be generalized Poisson it has to be either an open set in the symplectic locus or fully contained in the type change locus. In the former case there is nothing to blow up. } of $\J_1$. Then the blow-up is generalized K\"ahler.
\end{cor}
\begin{proof} Let $X_1$ be the type change locus for $\J_1$. In a local chart of the form (\ref{13:19}), $X_1$ is given by the vanishing of the holomorphic function $\sigma^{k/2}$ and as such is either empty or a codimension $1$ analytic subset of $\C^n$. We assume $X_1\neq \emptyset$, otherwise the statement is vacuous. Let $D'\subset M$ denote the Poisson subvariety of points where $Q$ does not assume its maximal rank on $M$. By Lemma \ref{14:37:27} and Equation (\ref{09:05}) we have
\begin{align*}
\text{ker}(Q)= \text{ker}(I^\ast_+-I^\ast_-)\oplus \text{ker}(I^\ast_++I^\ast_-)=\text{ker}(\pi_{\J_1})\oplus \text{ker}(\pi_{\J_2}).
\end{align*}
Consequently, $D'=X_1\cup X_2$, where $X_i$ is the set of points where $\pi_{\J_i}$ is not of maximal rank (or equivalently, where $\J_i$ is not of minimal type). Let $D$ be the union of the codimension $1$ components of $D'$. Then $D$ is also a Poisson subvariety which a priori could be empty, but we claim that $X_1\subset D$. Indeed, if $x\in X_1\backslash D$, then a neighborhood $U$ of $x$ in $X_1$ is disjoint from $D$. However, since $U$ is given by the vanishing of a holomorphic function (for a complex structure which need not coincide with either $I_\pm$), an open dense set in $U$ is a smooth submanifold of $M$ of real dimension $2n-2$. But $U\subset D'$ , and a real $2n-2$ dimensional submanifold of a complex manifold can not be contained in a finite union of analytic subsets of complex codimension bigger than $1$. So indeed $X_1\subset D$ and Theorem \ref{10:31} $ii)$ applies.
\end{proof}
\noindent A special case of this corollary is when $Y$ is a point and $M$ is $4$--dimensional. This situation was considered in \cite{MR2861778}, where it was assumed that the type change locus was smooth at the point in question. Note that if the type change locus is not smooth at this point, we are in situation $i)$ of Theorem \ref{10:31} and we can still blow it up. 
\begin{rem}
In \cite{MR2669364} Goto proved that every compact K\"ahler manifold with a holomorphic Poisson structure $\sigma$ has a generalized K\"ahler structure, where $\J_1$ is given by the Poisson deformation of the complex structure and $\J_2$ is a suitable deformation of the symplectic structure. If $Y$ is a complex Poisson submanifold for $\sigma$ whose conormal bundle is degenerate, then $\sigma$ lifts to the complex blow-up of $Y$, which has a K\"ahler metric of its own(cf.\ Remark \ref{13:56:55} $i)$). Applying Goto's theorem again, we see that the blow-up is again generalized K\"ahler and the blow-down map is generalized holomorphic with respect to $\J_1$. This example shows that Theorem \ref{10:31} is not optimal in its assumptions. However, the proof of Goto's theorem relies on the use of Green's operators for finding the right deformation of $\J_2$, and as such is non-local. In fact, since the K\"ahler metric on the blow-up differs from the original metric on a neighborhood of the exceptional divisor, there is a-priori nothing we can say about the relation between $\J_2$ on the blow-up and on the original manifold, even far away from the exceptional divisor. As such, it is not clear how to connect this example with the techniques used to prove Theorem \ref{10:31}.  
\end{rem}

\section{An example: compact Lie groups}
\label{09:21:55}

One source of examples of generalized K\"ahler manifolds is provided by Lie theory. 
\begin{prop}[\cite{gualtieri-2010}] 
Let $G$ be an even dimensional compact 
Lie group. Then $G$ has a generalized K\"ahler structure.
\end{prop}       
\noindent In order to find suitable submanifolds to blow up, we need to understand these generalized K\"ahler structures in some detail. Let $G$ be any Lie group. An element $\xi\in \mfg$ defines left and right invariant vector fields $\xi^L_g:=d_eL_g(\xi)$ and $\xi^R_g:=d_eR_g(\xi)$, and we have $[\xi^L,\eta^L]=[\xi,\eta]^L$, $[\xi^R,\eta^R]=-[\xi,\eta]^R$. These two trivializations of $TG$ define connections $\nabla^\pm$, characterized by $\nabla^+ \xi^L=0=\nabla^-\xi^R$ for $\xi\in \mfg$. Their torsion is given by 
\begin{align*}
T^+(\xi^L,\eta^L)=-[\xi,\eta]^L, \ \ \ \ \ \ \  T^-(\xi^R,\eta^R)=[\xi,\eta]^R.
\end{align*}
Suppose now that $G$ is compact and let $\langle, \rangle$ be a metric on $\mfg$ which is invariant under the adjoint action. In particular, its left and right invariant extensions over $G$ coincide and we denote this common extension by the same symbol $\langle, \rangle$. The $3$--form on $\mfg$ defined by $H(\xi,\eta,\zeta):=\langle [\xi,\eta],\zeta\rangle$ is also invariant under the adjoint action and so extends to a bi-invariant $3$--form on $G$. From the Jacobi identity it follows that $H$ is closed\footnote{In fact, since $H$ is constant in both trivializations it is parallel with respect to both $\nabla^\pm$ and therefore also for their affine combination $\nabla:=\frac{1}{2}(\nabla^++\nabla^-)$, which is nothing but the Levi-Cevita connection for $\langle, \rangle$.} and we have $\langle T^\pm(X,Y),Z\rangle=\mp H(X,Y,Z)$, so $\nabla^\pm$ coincide with the connections defined in Proposition \ref{11:28:39}.

Since $G$ is compact it is automatically reductive, i.e.\ its Lie algebra splits as $\mfg=\mfa\oplus \mfg'$, with $\mfa$ Abelian and $\mfg'$ semi-simple. Let $\mft$ be a maximal Cartan subalgebra of $\mfg'$ and $\mfg'_\C=\mft_\C\oplus \sum_{\alpha\in R}\mfg'_\alpha$ the associated root space decomposition. Since $G$ is compact the roots are contained in $i\mft\subset\mft^\ast_\C$, hence they come in pairs $\pm\alpha$ and we have $\overline{\mfg'_\alpha}=\mfg'_{-\alpha}$. Consequently, $\text{dim}(\mfg')$\ and $\text{dim}(\mft)$\ have the same parity and since $\mfg$ is even dimensional it follows that $\mfa\oplus\mft$ is even dimensional. Now choose a decomposition $R=R^-\cup R^+$ into positive and negative roots, so that in particular $-R_+=R_-$. We define a complex structure $I$ on $\mfg$ by the decomposition $\mfg_\C=\mfg^{1,0}\oplus \mfg^{0,1}$, where 
\begin{align*}
\mfg^{1,0}=(\mfa\oplus\mft)^{1,0}\oplus \sum_{\alpha\in R^+} \mfg'_\alpha,
\end{align*} 
and $(\mfa\oplus\mft)^{1,0}$ is defined with respect to an arbitrary but fixed complex structure on $\mfa\oplus\mft$, compatible with $\langle,\rangle$. By invariance of the metric, $\mfg_\alpha$ is orthogonal to $\mfg_\beta$ unless $\alpha=-\beta$, hence $I$ is compatible with $\langle, \rangle$. Since the sum of two positive roots is again positive it follows that $[\mfg^{1,0},\mfg^{1,0}]\subset \mfg^{1,0}$, so the complex structures $I_+$ and $I_-$, which are defined by the left respectively right invariant extensions of $I$ over $G$, are integrable. Since they are constant in the two respective trivializations, we have $\nabla^\pm I_\pm=0$, so $(G,\langle, \rangle, I_\pm, H)$ is generalized K\"ahler by Proposition \ref{11:28:39}. 
\newline
\newline
Next we look for generalized Poisson submanifolds for $\J_1$. A natural candidate is given by the complex locus of $\J_1$, i.e.\ the set of points where $I_+=I_-$ (see (\ref{11:28:39})). For $g\in G$ we have $(I_+)_g=(L_g)_\ast I$ and $(I_-)_g=(R_g)_\ast I$, so that $(I_+)_g=(I_-)_g$ if and only if $\text{Ad}(g)_\ast I=I$. This condition defines a subgroup of $G$ whose Lie algebra is given by 
\begin{align*}
\{\xi \in \mfg | \ [\text{ad}(\xi),I]=0 \Leftrightarrow [\xi,\mfg^{1,0}]\subset \mfg^{1,0}\}.
\end{align*}
One readily verifies that this algebra coincides with $\mfa\oplus \mft$,
hence the connected component of the complex locus that contains the identity equals the connected subgroup $T$ whose Lie algebra is $\mfa\oplus \mft$. Thus $T$, or any complex submanifold $Y\subset T$ for that matter, is a generalized Poisson submanifold of $(G,\J_1)$. To blow up $Y$ in $G$ with respect to $\J_1$, we need to understand the induced Lie algebra structure on $N^\ast Y$. Since $Y\subset T$, we have an inclusion of Lie algebras
\begin{align}\label{16:37:46}
N^\ast T|_Y\subset N^\ast Y\subset T^\ast G|_Y.
\end{align}
The action of $T$ on $G$, either from the left or the right, is a symmetry of the whole generalized K\"ahler structure that preserves $T$. In particular, the Lie brackets on $T^\ast_yG$ and $N^\ast_yT$ are independent of $y\in Y$ and we can compute them at $e\in G$. From (\ref{09:05}) we see that 
\begin{align*}
(\pi_{\J_1})_g&=-\frac{1}{2}(R_g)_\ast\big(\text{Ad}(g)_\ast \omega^{-1}-\omega^{-1}\big)  =-\frac{1}{2}(R_g)_\ast \Big( \text{Ad}(g)\circ \omega^{-1}\circ \text{Ad}(g)^\ast -\omega^{-1}\Big),
\end{align*}
where $\omega(\xi,\eta)=\langle I\xi,\eta\rangle$ is the associated Hermitian two--form on $\mfg$. Consequently, 
\begin{align*}
(\mcL_{\zeta^L}\pi_{\J_1})_e=\left.\frac{d}{dt}\right |_{t=0}\big(R_{exp(t\zeta)}^\ast (\pi_{\J_1})\big)_e=-\frac{1}{2} \Big( \text{ad}(\zeta)\circ \omega^{-1}+\omega^{-1} \circ \text{ad}(\zeta)^\ast	\Big)
\end{align*}
for $\zeta\in \mfg$. Let $\xi,\eta\in \mfg$ and denote by $\tilde{\xi},\tilde{\eta}\in \mfg^\ast$ their images under the metric. We have 
\begin{align*}
[\tilde{\xi},\tilde{\eta}]_{\pi_{{\J_1}}}(\zeta)=(\mcL_{\zeta^L}\pi_{\J_1})_e(\tilde{\xi},\tilde{\eta})=\frac{1}{2}\langle [I\xi,\eta]+[\xi,I\eta],\zeta 	\rangle.	
\end{align*}
for all $\zeta\in \mfg$. Hence, using the metric, the bracket $[,]_{\pi_{\J_1}}$ induces the following bracket on $\mfg$:
\begin{align*}
[\xi,\eta]_{{}_1}:=\frac{1}{2} \big([I\xi,\eta]+[\xi,I\eta]\big).
\end{align*}
Write $\xi=\xi'+\sum_\alpha(\xi_\alpha+\overline{\xi_\alpha})$, where $\xi'\in \mfa\oplus \mft$ and $\sum_\alpha(\xi_\alpha+\overline{\xi_\alpha}) \in \sum_\alpha (\mfg_\alpha\oplus\overline{\mfg_\alpha})_\R$, and similarly for $\eta$. Here and below, the summation on $\alpha$ is over all positive roots. Then
\begin{align}
\label{16:10:07}
[\xi,\eta]_{{}_1}=i\sum_{\alpha,\beta}\Big( [\xi_\alpha,\eta_\beta]-[\overline{\xi_\alpha},\overline{\eta_\beta}]
+\beta^{1,0}(\xi')\eta_\beta +\beta^{0,1}(\xi')\overline{\eta_\beta}-\alpha^{1,0}(\eta')\xi_\alpha-\alpha^{0,1}(\eta')\overline{\xi_\alpha}	\Big).
\end{align}
Here we are regarding the roots $\alpha,\beta \in (\mfa_\C\oplus \mft_\C)^\ast$ by extending them trivially over $\mfa_\C$, and define their $(1,0)$ and $(0,1)$ components with respect to $I|_{\mfa\oplus\mft}$. From (\ref{16:37:46}) we see that if $N^\ast_y Y$ is degenerate then so is $N^\ast_yT$, and therefore the whole of $N^\ast T$ by $T$-equivariance. So a necessary condition to blow up anything in $T$ is that $T$ itself can be blown up. This can be seen by restricting (\ref{16:10:07}) to $(T_eT)^\perp=\sum_\alpha (\mfg_\alpha\oplus\overline{\mfg_\alpha})_\R\subset \mfg$. There (\ref{16:10:07}) reduces to 
\begin{align*}
[\xi,\eta]_{{}_1}=i\sum_{\alpha,\beta}\big( [\xi_\alpha,\eta_\beta]-[\overline{\xi_\alpha},\overline{\eta_\beta}]\big).
\end{align*}
Recall that a Lie algebra is degenerate if and only if the bracket of any two elements lies in the plane spanned by them. In particular, as $[\mfg_\alpha,\mfg_\beta]=\mfg_{\alpha+\beta}$ for root decompositions, $N^\ast_eT\cong (T_eT)^\perp$ is degenerate if and only if the sum of any two positive roots is not a root itself. The only root systems satisfying this are products of $A_1$, corresponding to the Lie group $SU(2)$. 
In conclusion, in order to blow up $T$ in $G$ with respect to $\J_1$, we need $G$ to be of the form $G=(S^1)^n\times (S^3)^m$ for $n+m$ even, with $T=(S^1)^n\times (S^1)^m$. We then still have some residual freedom in choosing the complex submanifold $Y\subset T$. Instead of determining precisely all $Y$ for which (\ref{16:10:07}) becomes degenerate let us give an easy example. If $m$ is even we can take the complex structure $I$ on $\mfg$ to be a product complex structure, i.e.\ we can assume that $I$ preserves the decomposition $\mfa\oplus\mft$. Then if $Y$ is of the form $Y'\times (S^1)^m$ with $Y'\subset (S^1)^n$ a complex submanifold, (\ref{16:10:07}) vanishes on $N^\ast Y$, because all roots vanish on $\mfa_\C$. If $m$ is odd, we can choose the complex structure to be a product on $(S^1)^{n-1}\times (S^1\times (S^3)^m)$. Then if $Y=Y'\times (S^1)^{1+m}$ with $Y'\subset (S^1)^{n-1}$ complex, (\ref{16:10:07}) again vanishes on $N^\ast Y$. Note that since in all these cases $N^\ast Y$ is Abelian, Theorem \ref{10:31} $i)$ applies and we obtain 
\begin{thm}
Let $G=(S^1)^n\times (S^3)^m$, where $n+m$ is even, and let $T=(S^1)^n\times (S^1)^m\subset G$ be a maximal torus. Equip $G$ with a generalized K\"ahler structure as above for which $T$ is generalized Poisson. If either   
\begin{itemize}
\item[i)] $m$ is even and $Y=Y'\times (S^1)^m$ with $Y'\subset (S^1)^n$ complex, 
\end{itemize}
or
\begin{itemize}
\item[ii)] $m$ is odd and $Y=Y'\times (S^1)^{m+1}$ with $Y'\subset (S^1)^{n-1}$ complex,
\end{itemize}
then $Y$ can be blown up to a generalized K\"ahler manifold.
\end{thm}

\bibliographystyle{hyperamsplain}
\bibliography{references}

\providecommand{\bysame}{\leavevmode\hbox to3em{\hrulefill}\thinspace}
\providecommand{\MR}{\relax\ifhmode\unskip\space\fi MR }
\providecommand{\MRhref}[2]{%
  \href{http://www.ams.org/mathscinet-getitem?mr=#1}{#2}
}
\providecommand{\href}[2]{#2}
\begin{thebibliography}{10}

\bibitem{MR1702248}
V.~Apostolov, P.~Gauduchon, and G.~Grantcharov, \emph{Bi-{H}ermitian structures
  on complex surfaces},
  \href{http://dx.doi.org/10.1112/S0024611599012058}{Proc. London Math. Soc.
  (3) \textbf{79} (1999)}, no.~2, 414--428.

\bibitem{MR3128977}
M.~Bailey, \emph{Local classification of generalized complex structures}, J.
  Differential Geom. \textbf{95} (2013), no.~1, 1--37.

\bibitem{GCblowups}
M.~Bailey, G.~R. Cavalcanti, and J.~L. Van~der Leer~Duran, \emph{Blow-ups in
  generalized complex geometry}, In preparation.

\bibitem{MR3098084}
H.~Bursztyn, \emph{A brief introduction to {D}irac manifolds}, Geometric and
  topological methods for quantum field theory, Cambridge Univ. Press,
  Cambridge, 2013, pp.~4--38.

\bibitem{MR2397619}
H.~Bursztyn, G.~R. Cavalcanti, and M.~Gualtieri, \emph{Generalized {K}\"ahler
  and hyper-{K}\"ahler quotients}, Poisson geometry in mathematics and physics,
  Contemp. Math., vol. 450, Amer. Math. Soc., Providence, RI, 2008, pp.~61--77.
  \href{http://arxiv.org/abs/arXiv:math/0702104v1 [math.DG]}{{\tt
  arXiv:math/0702104v1 [math.DG]}}.

\bibitem{BCGinstantons}
\bysame, \emph{Generalized {K}\"ahler geometry of instanton moduli space},
  2012. \href{http://arxiv.org/abs/ArXiv:1203.2385}{{\tt ArXiv:1203.2385}}.

\bibitem{MR2574746}
G.~R. Cavalcanti and M.~Gualtieri, \emph{Blow-up of generalized complex
  4-manifolds}, \href{http://dx.doi.org/10.1112/jtopol/jtp031}{J. Topol.
  \textbf{2} (2009)}, no.~4, 840--864.

\bibitem{MR2861778}
\bysame, \emph{Blowing up generalized {K}\"ahler 4-manifolds},
  \href{http://dx.doi.org/10.1007/s00574-011-0028-1}{Bull. Braz. Math. Soc.
  (N.S.) \textbf{42} (2011)}, no.~4, 537--557.

\bibitem{MR2496412}
A.~Fino and A.~Tomassini, \emph{Non-{K}\"ahler solvmanifolds with generalized
  {K}\"ahler structure}, J. Symplectic Geom. \textbf{7} (2009), no.~2, 1--14.

\bibitem{MR776369}
S.~J. Gates, Jr., C.~M. Hull, and M.~Ro{\v{c}}ek, \emph{Twisted multiplets and
  new supersymmetric nonlinear {$\sigma$}-models},
  \href{http://dx.doi.org/10.1016/0550-3213(84)90592-3}{Nuclear Phys. B
  \textbf{248} (1984)}, no.~1, 157--186.

\bibitem{MR2669364}
R.~Goto, \emph{Deformations of generalized complex and generalized {K}\"ahler
  structures}, J. Differential Geom. \textbf{84} (2010), no.~3, 525--560.

\bibitem{gualtieri-2003}
M.~Gualtieri, \emph{Generalized complex geometry}, Ph.D. thesis, Oxford
  University, 2003. \href{http://arxiv.org/abs/ArXiv/0401221}{{\tt
  ArXiv/0401221}}.

\bibitem{gualtieri-2010}
\bysame, \emph{Generalized {K}\"ahler geometry}, 2010.
  \href{http://arxiv.org/abs/ArXiv:1007.3485}{{\tt ArXiv:1007.3485}}.

\bibitem{MR2811595}
\bysame, \emph{Generalized complex geometry},
  \href{http://dx.doi.org/10.4007/annals.2011.174.1.3}{Ann. of Math. (2)
  \textbf{174} (2011)}, no.~1, 75--123.

\bibitem{MR2217300}
N.~Hitchin, \emph{Instantons, {P}oisson structures and generalized {K}\"ahler
  geometry}, \href{http://dx.doi.org/10.1007/s00220-006-1530-y}{Comm. Math.
  Phys. \textbf{265} (2006)}, no.~1, 131--164,
  \href{http://arxiv.org/abs/arXiv:math/0503432v1 [math.DG]}{{\tt
  arXiv:math/0503432v1 [math.DG]}}.

\bibitem{MR2371181}
\bysame, \emph{Bihermitian metrics on del {P}ezzo surfaces}, J. Symplectic
  Geom. \textbf{5} (2007), no.~1, 1--8.

\bibitem{MR0212575}
B.~Malgrange, \emph{Ideals of differentiable functions}, Tata Institute of
  Fundamental Research Studies in Mathematics, No. 3, Tata Institute of
  Fundamental Research, Bombay; Oxford University Press, London, 1967.

\bibitem{MR1465521}
A.~Polishchuk, \emph{Algebraic geometry of {P}oisson brackets},
  \href{http://dx.doi.org/10.1007/BF02399197}{J. Math. Sci. (New York)
  \textbf{84} (1997)}, no.~5, 1413--1444. Algebraic geometry, 7.

\bibitem{MR1467652}
M.~Pontecorvo, \emph{Complex structures on {R}iemannian four-manifolds},
  \href{http://dx.doi.org/10.1007/s002080050108}{Math. Ann. \textbf{309}
  (1997)}, no.~1, 159--177.

\end{thebibliography}

\end{document}